\documentclass[1pt,leqno,onefignum,onetabnum]{siamltex1213}
\usepackage[leqno]{amsmath}
\usepackage{graphicx}
\usepackage{subfigure}
 \usepackage{threeparttable}

\usepackage{bm}
\usepackage{amssymb,version}
\usepackage{cases}
\newtheorem{rem}{Remark}[section]
\allowdisplaybreaks

\renewcommand\arraystretch{1.5}

    \title{On a SAV-MAC scheme for the Cahn-Hilliard-Navier-Stokes Phase Field Model
\thanks{The work of X. Li is supported by the Postdoctoral Science Foundation of China Grant No. BX20190187. The work of J. Shen is supported in part by NSF grants  DMS-1620262, 
DMS-1720442 and AFOSR  grant FA9550-16-1-0102.}
}
    \author{ Xiaoli Li
        \thanks{School of Mathematical Sciences and Fujian Provincial Key Laboratory on Mathematical Modeling and High Performance Scientific Computing, Xiamen University, Xiamen, Fujian, 361005, China. Email: xiaolisdu@163.com}.
        \and Jie Shen 
         \thanks{Corresponding Author. Department of Mathematics, Purdue University, West Lafayette, IN 47907, USA. Email: shen7@purdue.edu}.
}

\begin{document}
\maketitle

\begin{abstract}
We construct a numerical scheme based on the scalar auxiliary variable (SAV) approach in time  and the MAC discretization in space  for the Cahn-Hilliard-Navier-Stokes phase field model, and carry out  stability and error analysis. The scheme is linear, second-order, unconditionally energy stable and  can be implemented very efficiently. 
We establish  second-order error estimates both in time and space for 
phase field variable, chemical potential, velocity and pressure  in different discrete norms. We also provide numerical experiments to verify our theoretical results and  demonstrate the robustness and accuracy of the our scheme.
\end{abstract}

 

\pagestyle{myheadings}
\thispagestyle{plain}
\markboth{XIAOLI LI AND JIE SHEN}{SAV-Mac scheme for CHNS phase-field model}
 \section{Introduction} 
 Interfacial dynamics in the mixture of different fluids, solids or gas has been one of the fundamental issues in many fields of science and engineering, particularly in materials science and fluid dynamics, see for instance, \cite{cahn1958free,cahn1959free,yue2004diffuse,shen2018scalar} and the references therein. In recent years the phase field (i.e. diffuse interface) methods, have been successfully used to approximate a variety of interfacial dynamics. The basic idea for the phase field methods is that the interface is represented as a thin transition layer between two phases \cite{van1979thermodynamic,chen2016efficient}.   
 
The phase field model  can be derived from an energy variational approach. Thus a 
crucial goal in  algorithm design is to preserve the energy law at the discrete level. A large number of numerical schemes that have been developed for phase field models. 
Among them, the convex splitting approach \cite{shen2012second,wang2011energy, hu2009stable} and stabilized linearly implicit  approach \cite{liu2007dynamics,shen2010numerical,xu2006stability,zhao2016energy} are two popular ways to constuct unconditionally energy stable schemes. Unfortunately, the convex splitting approach usually leads to nonlinear schemes, and the stabilized linearly implicit  approach results in additional accuracy issues and may not be easy to obtain second order unconditionally energy stable schemes. Recently, a novel numerical method of invariant energy quadratization (IEQ), has been proposed in \cite{cheng2017efficient,zhao2017novel,yang2017numerical}. This method is a generalization of the method of Lagrange multipliers or of auxiliary variable. The IEQ approach is remarkable as it permits us to construct linear and second-order unconditionally energy stable schemes for a large class of gradient flows.
However, it  leads to coupled systems with time-dependent variable coefficients. The scalar auxiliary variable (SAV) approach \cite{shen2018scalar,shen2017new}  inherits  advantages of the IEQ approach but leads to decoupled systems with constant coefficients so it is both accurate and very efficient.

 As for the Cahn-Hilliard-Navier-Stokes phase-field models, Shen and Yang  \cite{shen2010phase,shen2015decoupled} constructed several efficient time discretization schemes for two-phase incompressible flows with different densities and viscosities,  established discrete energy laws but  no error estimates were derived. Second order in time numerical scheme based on the 
convex-splitting for the Cahn-Hilliard equation and pressure-projection for the Navier-Stokes equation has been constructed by Han and Wang in \cite{han2015second}.
With regards to the numerical analysis, Feng, He and Liu \cite{feng2007analysis} proposed and analyzed some semi-discrete and fully discrete finite element schemes with the abstract convergence by making use of the discrete energy law. Gr\"un \cite{grun2013convergent} proved a abstract convergence result of a fully discrete scheme for a diffuse interface models for two-phase incompressible fluids. Diegel, Feng, and Wise \cite{diegel2015analysis} developed a fully discrete mixed finite element convex-splitting scheme for the Cahn-Hilliard-Darcy-Stokes system. The time discretization used is a first-order implicit Euler. They proved unconditional energy stability and error estimates for the phase field variable, chemical potential and velocity. No convergence rate for pressure was demonstrated in their work.

The work presented in this paper for the Cahn-Hilliard-Navier-Stokes phase field model is unique in the following aspects. First, we construct fully discrete linear,  second-order  (in space and time), unconditionally energy stable scheme for the Cahn-Hilliard-Navier-Stokes phase field model. Furthermore, the scheme can be very efficiently implemented.  Secondly, 
 we carry out a rigorous error analysis to derive second-order error estimates  both in time and space  for 
phase field variable, chemical potential, velocity and pressure in different discrete norms for the Cahn-Hilliard-Stokes phase field model.  We believe that this is the first such result for any  fully discrete linear schemes for  Cahn-Hilliard-Stokes or Cahn-Hilliard-Navier-Stokes models without assuming a uniform Lipschitz condition on the nonlinear potential.


The paper is organized as follows. In Section 2 we describe  the problem and present some notations. In Section 3 we present  the fully discrete SAV-MAC schemes and prove their stability. In Section 4 we carry out  error estimates for the fully discrete SAV-MAC scheme for the Cahn-Hilliard-Stokes system. In Section 5, we present some numerical experiments  to verify the accuracy of the proposed numerical schemes. More details about  the MAC scheme are given in the Appendix.
\section{The Problem Description and  Notations} \label{sec:Notation}
We consider  the following incompressible Cahn-Hilliard-Navier-Stokes phase field model (cf. \cite{feng2007analysis,chen2016efficient,diegel2015analysis}):

  \begin{subequations}\label{e_model}
    \begin{align}
    \frac{\partial \phi}{\partial t}=M\Delta \mu-\textbf{u}\cdot\nabla\phi  \quad &\ in\ \Omega\times J,
    \label{e_modelA}\\
    \mu=-\lambda\Delta \phi+\lambda F^{\prime}(\phi) \quad &\ in\ \Omega\times J,
    \label{e_modelB}\\
     \frac{\partial \textbf{u}}{\partial t}+\gamma \textbf{u}\cdot \nabla\textbf{u}
     -\nu\Delta\textbf{u}+\nabla p=\mu\nabla\phi
     \quad &\ in\ \Omega\times J,
      \label{e_modelC}\\
      \nabla\cdot\textbf{u}=0
      \quad &\ in\ \Omega\times J,
      \label{e_modelD}\\
       \frac{\partial \phi}{\partial \textbf{n}}= \frac{\partial \mu}{\partial \textbf{n}}=0,\
     \textbf{u}=\textbf{0}
      \quad &\ on\ \partial\Omega\times J,
      \label{e_modelE}
    \end{align}
  \end{subequations}
where $\displaystyle F(\phi)=\frac{1}{4\epsilon^2}(1-\phi^2)^2$, $M>0$ is the mobility constant, $\nu>0$ is the fluid viscosity. $\lambda>0$ is the mixing coefficient, $\Omega$ is a bounded domain and   $J=(0,T]$.  The unknowns are the velocity $\textbf{u}$, the pressure $p$, the pase function $\phi$ and the chemical potential $\mu$. It models the dynamics of the mixture of two-incompressible fluids with the same density, which is set to be $\rho_0=1$ for simplicity. $\gamma$ is an additional parameter that we added to distinguish the Cahn-Hilliard-Navier-Stokes model ($\gamma=1$) and the Cahn-Hilliard-Stokes model ($\gamma=0$). When the viscosity $\nu$ is not sufficient large,  the Cahn-Hilliard-Stokes model can be used as a good approximation to the Cahn-Hilliard-Navier-Stokes model.

Taking the inner products of \eqref{e_modelA} with $\mu$, \eqref{e_modelB}  with $\frac{\partial \phi}{\partial t}$, \eqref{e_modelC}  with $\textbf{u}$  respectively, we obtain the following energy dissipation law:
\begin{equation}\label{energy1}
 \frac{d E(\phi,\textbf{u})}{d t}=-M\|\nabla \mu\|^2-\nu\|\nabla\textbf{u}\|^2,
\end{equation}
where $E(\phi,\textbf{u})=\int_\Omega\{\frac 12|\textbf{u}|^2+\lambda (\frac 12 |\nabla \phi|^2+F(\phi))\}$ is the total energy.

For two-phase flows  with low Reynolds numbers, one can approximate 

We now introduce some standard notations.

Let $L^m(\Omega)$ be the standard Banach space with norm
$$\| v\|_{L^m(\Omega)}=\left(\int_{\Omega}| v|^md\Omega\right)^{1/m}.$$
For simplicity, let
$$(f,g)=(f,g)_{L^2(\Omega)}=\int_{\Omega}fgd\Omega$$
denote the $L^2(\Omega)$ inner product,
 $\|v\|_{\infty}=\|v\|_{L^{\infty}(\Omega)}.$ And $W_p^k(\Omega)$ be the standard Sobolev space
$$W_p^k(\Omega)=\{g:~\| g\|_{W_p^k(\Omega)}<\infty\},$$
where
\begin{equation}\label{enorm1}
\| g\|_{W_p^k(\Omega)}=\left(\sum\limits_{|\alpha|\leq k}\| D^\alpha g\|_{L^p(\Omega)}^p \right)^{1/p}.
\end{equation}

Throughout the paper we use $C$, with or without subscript, to denote a positive
constant, independent of discretization parameters, which could have different values at different places.

  \section{The SAV Schemes and their stability}
 In this section, we first reformulate the phase-field system into an equivalent 
  system with an additional scalar auxiliary variable (SAV). Then,  
   we construct  semi discrete and fully discrete SAV schemes, and prove that they are unconditionally energy stable.
  
 \subsection{The SAV reformulation} 
 We introduce  a scalar auxiliary variable $r(t)=\sqrt{E_1(\phi)+\delta}$ with any $\delta>0$, and reformulate the  
 system (\ref{e_model})  as:
 \begin{subequations}\label{e_model_r}
    \begin{align}
    \frac{\partial \phi}{\partial t}=M\Delta \mu-\textbf{u}\cdot\nabla\phi  \quad &\ in\ \Omega\times J,
    \label{e_model_rA}\\
    \mu=-\lambda\Delta \phi+\lambda\frac{r}{\sqrt{E_1(\phi)+\delta}}F^{\prime}(\phi) \quad &\ in\ \Omega\times J,
    \label{e_model_rB}\\
       r_t=\frac{1}{2\sqrt{E_1(\phi)+\delta}}\int_{\Omega}F^{\prime}(\phi)\phi_t d\textbf{x}\quad &\ in\ \Omega\times J,
    \label{e_model_rC}\\   
     \frac{\partial \textbf{u}}{\partial t}+\gamma\textbf{u}\cdot \nabla\textbf{u}
     -\nu\Delta\textbf{u}+\nabla p=\mu\nabla\phi
     \quad &\ in\ \Omega\times J,
      \label{e_model_rD}\\
      \nabla\cdot\textbf{u}=0
      \quad &\ in\ \Omega\times J.
      \label{e_model_rE}
    \end{align}
  \end{subequations}
  where $E_1(\phi)=\int_{\Omega}F(\phi)d\textbf{x}$. It is clear that with $r(0)=\sqrt{E_1(\phi|_{t=0})+\delta}$, the above system is equivalent to (\ref{e_model}).  
  Taking the inner products of \eqref{e_model_rA} with $\mu$, \eqref{e_model_rB}  with $\frac{\partial \phi}{\partial t}$, \eqref{e_model_rC}  with $2\lambda r$ and  \eqref{e_model_rD} with $\textbf{u}$  respectively, we obtain the following energy dissipation law:
\begin{equation}\label{energy1_modify}
 \frac{d \tilde E(\phi,\textbf{u},r)}{d t}=-M\|\nabla \mu\|^2-\nu\|\nabla\textbf{u}\|^2,
\end{equation}
where $\tilde E(\phi,\textbf{u},r)=\int_\Omega\frac 12\{|\textbf{u}|^2+\lambda |\nabla \phi|^2\}d\textbf{x}+\lambda r^2$ is the total energy.
  
  \subsection{The semi discrete SAV/CN scheme}
 
  Set  $\Delta t=T/N,~t^n=n\Delta t, ~for~n\leq N,$
and define
$$[d_{t}f]^n=\frac{f^n-f^{n-1}}{\Delta t},\ \ f^{n+1/2}=\frac{f^n+f^{n+1}}{2}.$$
Then, a second-order SAV scheme based on Crank-Nicolson is:
\begin{subequations}\label{e_model_r_time_discrete}
    \begin{align}
 &    \frac{\phi^{n+1}-\phi^n}{\Delta t}=M\Delta \mu^{n+1/2}-\textbf{u}^{n+1/2}\cdot
     \nabla\tilde{\phi}^{n+1/2}, \label{e_model_r_time_discrete1}\\
& \mu^{n+1/2}=-\lambda\Delta \phi^{n+1/2}+\lambda\frac{r^{n+1/2}}{\sqrt{E_1(\tilde{\phi}^{n+1/2})+\delta}}F^{\prime}(\tilde{\phi}^{n+1/2}),\label{e_model_r_time_discrete2}\\
& \frac{r^{n+1}-r^n}{\Delta t}=\frac{1}{2\sqrt{E_1(\tilde{\phi}^{n+1/2})+\delta}}\int_{\Omega}F^{\prime}(\tilde{\phi}^{n+1/2}) \frac{\phi^{n+1}-\phi^n}{\Delta t} d\textbf{x},\label{e_model_r_time_discrete3}\\
& \frac{\textbf{u}^{n+1}-\textbf{u}^n}{\Delta t}+\gamma\tilde{\textbf{u}}^{n+1/2}\cdot \nabla\textbf{u}^{n+1/2}  -\nu\Delta\textbf{u}^{n+1/2}\nonumber\\
&\hskip   3cm   +\nabla p^{n+1/2}=\mu^{n+1/2}\nabla\tilde{\phi}^{n+1/2},\label{e_model_r_time_discrete4}\\
&  \nabla\cdot\textbf{u}^{n+1/2}=0,  \label{e_model_r_time_discrete5}
     \end{align}
  \end{subequations}
 where $\tilde{\textbf{u}}^{n+1/2}=(3\textbf{u}^n-\textbf{u}^{n-1})/2$ and  $\tilde{\phi}^{n+1/2}=(3\phi^n-\phi^{n-1})/2$. We also set  $\textbf{u}^{-1}=\textbf{u}^{0}$.

\begin{theorem}
 The scheme \eqref{e_model_r_time_discrete} is unconditionally energy stable in the sense that
 \begin{equation*}
 \tilde E^{n+1}(\phi,\textbf{u},r)-\tilde E^n(\phi,\textbf{u},r)=-M\|\nabla \mu^{n+1/2}\|^2-\nu\|\nabla \textbf{u}^{n+1/2}\|^2,
\end{equation*}
where 
\begin{equation*}
 \tilde E^{n+1}(\phi,\textbf{u},r)=\int_\Omega\frac 12\{|\textbf{u}^{n+1}|^2+\lambda  |\nabla \phi^{n+1}|^2\}d\textbf{x}+\lambda |r^{n+1}|^2.
 \end{equation*}
\end{theorem}
\begin{proof}
 The proof is quite straightforward.  Taking the inner products of \eqref{e_model_r_time_discrete1} with $\mu^{n+\frac 12}$, \eqref{e_model_r_time_discrete2}  with $\frac{\phi^{n+1}-\phi^n}{\Delta t}$, \eqref{e_model_r_time_discrete3}   with $2\lambda r^{n+1/2}$ and  \eqref{e_model_r_time_discrete4}  with $\textbf{u}^{n+1/2}$  respectively, we obtain immediately the desired result.
\end{proof}

 \begin{rem} \label{rem1}
\begin{itemize}
  \item The above scheme is second-order in time and linear, but it is weakly coupled. The above stability result indicates that this weakly coupled system is positive definite. 

\item  If  $\textbf{u}^{n+1/2}$ in \eqref{e_model_r_time_discrete1} is replaced by an explicit second-order extrapolation,  
 $(\phi^{n+1},\mu^{n+1},r^{n+1})$ can be obtained  from  \eqref{e_model_r_time_discrete1}-\eqref{e_model_r_time_discrete3}  efficiently by solving decoupled elliptic systems with constant coefficients (cf. \cite{shen2018scalar}).   Once $\mu^{n+1}$ is known, we can solve $(\textbf{u}^{n+1}, p^{n+1})$ from 
 \eqref{e_model_r_time_discrete4}-\eqref{e_model_r_time_discrete5} which is essentially a generalized Stokes problem that can be solved efficiently with a MAC scheme (see below). 
\item We can use the decoupled  scheme with explicit treatment of $ \textbf{u}^{n+1/2}$ in \eqref{e_model_r_time_discrete1}  as a preconditioner for the weakly coupled scheme.
\end{itemize}
\end{rem}

  \subsection{Spacial discretization by finite differences}

Denote by $\{Z^n, W^n, R^n,$
$ \textbf{U}^n, P^n\}_{n=1}^{N}$, the approximations to $\{\phi^n, \mu^n, r^n, \textbf{u}^n, p^n\}_{n=1}^{N}$ respectively, with  the boundary conditions
 \begin{eqnarray}\label{boundary condition}
 \left\{
 \begin{array}{lll}
 \displaystyle [D_xZ]_{0,j+1/2}^{n}=[D_xZ]_{N_x,j+1/2}^{n}=0,& 0\leq j\leq N_y-1,\\
 \displaystyle [D_yZ]_{i+1/2,0}^{n}=[D_yZ]_{i+1/2,N_y}^{n}=0,& 0\leq i\leq N_x-1,\\
 \displaystyle [D_xW]_{0,j+1/2}^{n}=[D_xW]_{N_x,j+1/2}^{n}=0,& 0\leq j\leq N_y-1,\\
 \displaystyle [D_yW]_{i+1/2,0}^{n}=[D_yW]_{i+1/2,N_y}^{n}=0,& 0\leq i\leq N_x-1,\\
 \displaystyle U_{1,0,j+1/2}^{n}=U_{1,N_x,j+1/2}^{n}=0,& 0\leq j\leq N_y-1,\\
 \displaystyle U_{1,i,0}^{n}=U_{1,i,N_y}^{n}=0,& 0\leq i\leq N_x,\\
 \displaystyle U_{2,0,j}^{n}=U_{2,N_x,j}^{n}=0,  &0\leq j\leq N_y,\\
   \displaystyle U_{2,i+1/2,0}^{n}=W_{2,i+1/2,N_y}^{n}=0,& 0\leq i\leq N_x-1,
  \end{array}
  \right.
 \end{eqnarray}
 and initial conditions
 \begin{eqnarray}\label{initial condition}
 \left\{
 \begin{array}{lll}   
  \displaystyle Z_{i+1/2,j+1/2}^0=\phi^0_{i+1/2,j+1/2}, &0\leq i\leq N_x-1,0\leq j\leq N_y-1,\\
  \displaystyle U_{1,i,j+1/2}^{0}=u^0_{1,i,j+1/2}, &0\leq i\leq N_x,0\leq j\leq N_y,\\
  \displaystyle U_{2,i+1/2,j}^{0}=u^0_{2,i+1/2,j}, &0\leq i\leq N_x,0\leq j\leq N_y,
  \end{array}
  \right.
 \end{eqnarray}
  where $\phi^0$, $\textbf{u}^0$ are given initial conditions respectively.

Then, the  fully discrete SAV/CN  scheme based on the MAC discretization is as follows:
\begin{subequations}\label{e_model_r_full_discrete}
    \begin{align}
   & [d_tZ]^{n+1}=M[d_xD_xW+d_yD_yW]^{n+1/2}
   -\mathcal{P}_h^y\mathcal{P}_h^x[U_1D_x\tilde{Z}+U_2D_y\tilde{Z}]^{n+1/2},  \label{e_model_r_full_discrete1}\\
 &   W^{n+1/2}=-\lambda[d_xD_xZ+d_yD_yZ]^{n+1/2}
    +\lambda\frac{R^{n+1/2}}{\sqrt{E_1^h(\tilde{Z}^{n+1/2})+\delta}}F^{\prime}(\tilde{Z}^{n+1/2}),\label{e_model_r_full_discrete2}\\
 & d_tR^{n+1}=\frac{1}{2\sqrt{E_1^h(\tilde{Z}^{n+1/2})+\delta}}(F^{\prime}(\tilde{Z}^{n+1/2}),
d_tZ^{n+1})_{l^2,M},\label{e_model_r_full_discrete3}\\
&[d_tU_1]^{n+1}+\frac{\gamma}{2}[\tilde{U}_1D_x(\mathcal{P}_h^xU_1)+\mathcal{P}_h^xd_x(U_1\tilde{U}_1)+\mathcal{P}_h^y(\mathcal{P}_h^x\tilde{U}_2D_yU_1)\notag \\
&\hskip 1cm +d_y(\mathcal{P}_h^yU_1\mathcal{P}_h^x\tilde{U}_2)]^{n+1/2}
-\nu D_x(d_xU_1)^{n+1/2}-\nu d_y(D_yU_1)^{n+1/2} \label{e_model_r_full_discrete4}\\
&\hskip 1cm +[D_xP]^{n+1/2}=\mathcal{P}_h^x W^{n+1/2}[D_x\tilde{Z}]^{n+1/2},\notag\\
& [d_tU_2]^{n+1}+\frac{\gamma}{2}[\mathcal{P}_h^x(\mathcal{P}_h^y\tilde{U}_1D_xU_2)
+d_x(\mathcal{P}_h^y\tilde{U}_1\mathcal{P}_h^xU_2)+\tilde{U}_2D_y(\mathcal{P}_h^yU_2)\notag \\
&\hskip 1cm +\mathcal{P}_h^y(d_y(U_2\tilde{U}_2))]^{n+1/2}
-\nu D_y(d_yU_2)^{n+1/2}-\nu d_x(D_xU_2)^{n+1/2}
\label{e_model_r_full_discrete5} \\
&\hskip 1cm +[D_yP]^{n+1/2}=\mathcal{P}_h^yW^{n+1/2}[D_y\tilde{Z}]^{n+1/2},\notag\\
&[d_xU_1]^{n+1/2}+[d_yU_2]^{n+1/2}=0,  \label{e_model_r_full_discrete6}
\end{align}
\end{subequations}
where $\mathcal{P}_h^x$ and $\mathcal{P}_h^y$ are linear interpolation operators in the $x$ and $y$ directions respectively,  and $\tilde{{H}}^{n+1/2}=\frac 32{H}^n-\frac 12{H}^{n-1}$ for any sequence  $\{\textbf{H}^k\}$.

\medskip
\begin{rem}
 The above scheme can be efficiently solved using the strategies described in Remark \ref{rem1}. Moreover,  thanks to the structure of the MAC scheme, if $W^{n+1/2}$ is known, the pressure $P^{n+1/2}$ can be decoupled from 
\eqref{e_model_r_full_discrete4}-\eqref{e_model_r_full_discrete5}
by solving a discrete pressure Poisson equation. Hence, the above scheme can be very efficiently implemented. 
\end{rem}


 It is easy to verify that the following discrete integration-by-part formulae hold.
 \begin{lemma}  \label{lemma:U-P-Relation}
\cite{weiser1988convergence} Let $\{V_{1,i,j+1/2}\},\{V_{2,i+1/2,j}\}$ and $\{q_{1,i+1/2,j+1/2}\},\{q_{2,i+1/2,j+1/2}\}$
be discrete functions  with
$V_{1,0,j+1/2}=V_{1,N_x,j+1/2}=V_{2,i+1/2,0}=V_{2,i+1/2,N_y}=0$, with proper integers $i$ and $j$.
 Then there holds
 \begin{equation}
 \left\{
  \begin{array}{lll}
        (D_x q_1,V_1)_{l^2,T,M}&=& -( q_1, d_x V_1)_{l^2,M},\\
        (D_y q_2,V_2)_{l^2,M,T}&=& -( q_2, d_y V_2)_{l^2,M}.
  \end{array}
  \right.
 \end{equation}
\end{lemma}

 \begin{theorem}\label{thm_discrete total energy}
The scheme \eqref{e_model_r_full_discrete1}-\eqref{e_model_r_full_discrete6} is unconditionally energy stable in the sense that
 \begin{equation*}
 \tilde E^{n+1}(Z,\textbf{U},R)-\tilde E^n(Z,\textbf{U},R)=-M\Delta t \|D W^{n+1/2}\|_{l^2}^2-\nu \Delta t\|D \textbf{U}^{n+1/2}\|_{l^2}^2,
\end{equation*}
where $\textbf{D}H=(D_xH,D_yH)$  for any discrete scalar or vector function $H$, and
\begin{equation*}
 \tilde E^{n+1}(Z,\textbf{U},R)=\frac 12\|\textbf{U}\|_{l^2}^2+\lambda(\frac 12\|DZ^{n+1}\|_{l^2}^2+(R^{n+1})^2).
 \end{equation*}
\end{theorem}

\begin{proof}  
Multiplying (\ref{e_model_r_full_discrete1}) by $W_{i+1/2,j+1/2}^{n+1/2}hk$, and making summation on $i,j$ for $0\leq i\leq N_x-1,\ 0\leq j\leq N_y-1$, we have 
\begin{equation}\label{e_Stability1}
\aligned
(d_tZ^{n+1},W^{n+1/2})_{l^2,M}=&M(d_xD_xW^{n+1/2}+d_yD_yW^{n+1/2},W^{n+1/2})_{l^2,M}\\
&-(\mathcal{P}_h^y\mathcal{P}_h^x[U_1D_x\tilde{Z}+U_2D_y\tilde{Z}]^{n+1/2},W^{n+1/2})_{l^2,M}.
\endaligned
\end{equation} 
Taking notice of Lemma \ref{lemma:U-P-Relation}, the first term on the right hand side of (\ref{e_Stability1}) can be transformed into the following:
\begin{equation}\label{e_Stability1_1}
\aligned
&M(d_xD_xW^{n+1/2}+d_yD_yW^{n+1/2},W^{n+1/2})_{l^2,M}\\
=&-M\|D_xW^{n+1/2}\|^2_{l^2,T,M}-M\|D_yW^{n+1/2}\|^2_{l^2,M,T}\\
=&-M\|\textbf{D}W^{n+1/2}\|_{l^2}.
\endaligned
\end{equation} 

Multiplying (\ref{e_model_r_full_discrete2}) by $d_tZ^{n+1}_{i+1/2,j+1/2}hk$, and making summation on $i,j$ for $0\leq i\leq N_x-1,~0\leq j\leq N_y-1$, we have 
\begin{equation}\label{e_Stability2}
\aligned
(d_tZ^{n+1},W^{n+1/2})_{l^2,M}=&-\lambda(d_xD_xZ^{n+1/2}+d_yD_yZ^{n+1/2},d_tZ^{n+1})_{l^2,M}\\
&+\lambda\frac{R^{n+1/2}}{\sqrt{E_1^h(\tilde{Z}^{n+1/2})+\delta}}
(F^{\prime}(\tilde{Z}^{n+1/2}), d_tZ^{n+1})_{l^2,M}.
\endaligned
\end{equation}
Recalling Lemma \ref{lemma:U-P-Relation}, the first term on the right hand side of (\ref{e_Stability2}) can be estimated by:
\begin{equation}\label{e_Stability2_1}
\aligned
-\lambda(d_xD_xZ^{n+1/2}&+d_yD_yZ^{n+1/2},d_tZ^{n+1})_{l^2,M}\\
=&\lambda(D_xZ^{n+1/2},d_tD_xZ^{n+1})_{l^2,T,M}+\lambda(D_yZ^{n+1/2},d_tD_yZ^{n+1})_{l^2,M,T}\\
=&\lambda\frac{\|\textbf{D}Z^{n+1}\|^2_{l^2}-\|\textbf{D}Z^n\|_{l^2}^2}{2\Delta t}.
\endaligned
\end{equation}
Multiplying equation (\ref{e_model_r_full_discrete3}) by $(R^{n+1}+R^{n})$ leads to 
\begin{equation}\label{e_Stability3}
\aligned
\frac{(R^{n+1})^2-(R^{n})^2}{\Delta t}=
\frac{R^{n+1/2}}{\sqrt{E_1^h(\tilde{Z}^{n+1/2})+\delta}}(F^{\prime}(\tilde{Z}^{n+1/2}),
d_tZ^{n+1})_{l^2,M}.
\endaligned
\end{equation}   
Combining (\ref{e_Stability3}) with (\ref{e_Stability1})-(\ref{e_Stability2_1}) gives that
\begin{equation}\label{e_Stability4}
\aligned
&\lambda\frac{(R^{n+1})^2-(R^{n})^2}{\Delta t}
+\lambda\frac{\|\textbf{D}Z^{n+1}\|^2_{l^2}-\|\textbf{D}Z^n\|_{l^2}^2}{2\Delta t}\\
=&-M\|\textbf{D}W^{n+1/2}\|_{l^2}^2-(\mathcal{P}_h^y\mathcal{P}_h^x[U_1D_x\tilde{Z}+U_2D_y\tilde{Z}]^{n+1/2},W^{n+1/2})_{l^2,M}.
\endaligned
\end{equation} 
Multiplying (\ref{e_model_r_full_discrete4}) by $U_{1,i,j+1/2}^{n+1/2}hk$, and making summation on $i,j$ for $1\leq i\leq N_x-1,\ 0\leq j\leq N_y-1$, we have 
\begin{equation}\label{e_Stability5}
\aligned
&(d_tU_1^{n+1},U_1^{n+1/2})_{l^2,T,M}+\frac{\gamma}{2}\left((\tilde{U}_1^{n+1/2}D_x(\mathcal{P}_h^xU_1^{n+1/2}),U_1^{n+1/2})_{l^2,T,M}\right.\\
&+(\mathcal{P}_h^xd_x(U_1^{n+1/2}\tilde{U}_1^{n+1/2}),U_1^{n+1/2})_{l^2,T,M}
+(\mathcal{P}_h^y(\mathcal{P}_h^x\tilde{U}_2^{n+1/2}D_yU_1^{n+1/2}),U_1^{n+1/2})_{l^2,T,M}\\
&+\left.(d_y(\mathcal{P}_h^yU_1^{n+1/2}\mathcal{P}_h^x\tilde{U}_2^{n+1/2}),U_1^{n+1/2})_{l^2,T,M}\right)+\nu\|d_x U^{n+1/2}_1\|^2_{l^2,M}\\
&+\nu\|D_yU^{n+1/2}_1\|^2_{l^2,T_y}-(P^{n+1/2},d_xU^{n+1/2}_1)_{l^2,M}\\
=&(\mathcal{P}_h^x W^{n+1/2}D_x\tilde{Z}^{n+1/2},U_1^{n+1/2})_{l^2,T,M}.
\endaligned
\end{equation} 
Thanks to Lemma \ref{lemma:U-P-Relation}, we have
\begin{equation}\label{e_Stability_Navier_Stokes1}
\aligned
(\tilde{U}_1^{n+1/2}&D_x(\mathcal{P}_h^xU_1^{n+1/2}),U_1^{n+1/2})_{l^2,T,M}\\
=&-(\mathcal{P}_h^xU_1^{n+1/2},d_x(\tilde{U}_1^{n+1/2}U_1^{n+1/2}))_{l^2,M}\\
=&-(\mathcal{P}_h^xd_x(\tilde{U}_1^{n+1/2}U_1^{n+1/2}),U_1^{n+1/2})_{l^2,T,M}.
\endaligned
\end{equation}  
The fifth term on the left hand side of (\ref{e_Stability5}) can be estimated as follows:
\begin{equation}\label{e_Stability_Navier_Stokes2}
\aligned
(d_y(\mathcal{P}_h^y U_1^{n+1/2}&\mathcal{P}_h^x\tilde{U}_2^{n+1/2}),U_1^{n+1/2})_{l^2,T,M}\\
=&-(\mathcal{P}_h^yU_1^{n+1/2}\mathcal{P}_h^x\tilde{U}_2^{n+1/2},D_yU_1^{n+1/2})_{l^2,M}\\
=&-(\mathcal{P}_h^y(\mathcal{P}_h^x\tilde{U}_2^{n+1/2}D_yU_1^{n+1/2}),U_1^{n+1/2})_{l^2,T,M}.
\endaligned
\end{equation} 
Multiplying (\ref{e_model_r_full_discrete5}) by $U_{2,i+1/2,j}^{n+1/2}hk$, and making summation on $i,j$ for $0\leq i\leq N_x-1,\ 1\leq j\leq N_y-1$, we can obtain
\begin{equation}\label{e_Stability6}
\aligned
&(d_tU_2^{n+1},U_2^{n+1/2})_{l^2,M,T}+\frac{\gamma}{2}\left((\mathcal{P}_h^x(\mathcal{P}_h^y\tilde{U}_1^{n+1/2}D_xU_2^{n+1/2}),U_2^{n+1/2})_{l^2,M,T}\right.\\
&+(d_x(\mathcal{P}_h^y\tilde{U}_1^{n+1/2}\mathcal{P}_h^xU_2^{n+1/2}),U_2^{n+1/2})_{l^2,M,T}+(\tilde{U}_2^{n+1/2}D_y(\mathcal{P}_h^yU_2^{n+1/2}),U_2^{n+1/2})_{l^2,M,T}\\
&\left.+(\mathcal{P}_h^y(d_y(U_2^{n+1/2}\tilde{U}_2^{n+1/2})),U_2^{n+1/2})_{l^2,M,T}\right)+\nu\|d_y U^{n+1/2}_2\|^2_{l^2,M}\\
&+\nu\|D_xU^{n+1/2}_2\|^2_{l^2,T_x}-(P^{n+1/2},d_yU^{n+1/2}_2)_{l^2,M}\\
=&(\mathcal{P}_h^y W^{n+1/2}D_y\tilde{Z}^{n+1/2},U_2^{n+1/2})_{l^2,M,T}.
\endaligned
\end{equation} 
Similar to the estimates of (\ref{e_Stability_Navier_Stokes1}) and (\ref{e_Stability_Navier_Stokes2}), we have
\begin{equation}\label{e_Stability_Navier_Stokes3}
\aligned
&(\mathcal{P}_h^x(\mathcal{P}_h^y\tilde{U}_1^{n+1/2}D_xU_2^{n+1/2}),U_2^{n+1/2})_{l^2,M,T}\\
&+(d_x(\mathcal{P}_h^y\tilde{U}_1^{n+1/2}\mathcal{P}_h^xU_2^{n+1/2}),U_2^{n+1/2})_{l^2,M,T}=0,
\endaligned
\end{equation} 
and 
\begin{equation}\label{e_Stability_Navier_Stokes4}
\aligned
&(\tilde{U}_2^{n+1/2}D_y(\mathcal{P}_h^yU_2^{n+1/2}),U_2^{n+1/2})_{l^2,M,T}\\
&+(\mathcal{P}_h^y(d_y(U_2^{n+1/2}\tilde{U}_2^{n+1/2})),U_2^{n+1/2})_{l^2,M,T}=0.
\endaligned
\end{equation}  
Combining (\ref{e_Stability5})-(\ref{e_Stability_Navier_Stokes4}) and recalling (\ref{e_model_r_full_discrete6}) lead to
\begin{equation}\label{e_Stability7}
\aligned
&\frac{\|\textbf{U}^{n+1}\|^2_{l^2}-\|\textbf{U}^{n}\|_{l^2}^2}{2\Delta t}
+\nu\|D\textbf{U}\|^2\\
=&(\mathcal{P}_h W^{n+1/2}D_x\tilde{Z}^{n+1/2},U_1^{n+1/2})_{l^2,T,M}
+(\mathcal{P}_h W^{n+1/2}D_y\tilde{Z}^{n+1/2},U_2^{n+1/2})_{l^2,M,T}.
\endaligned
\end{equation} 
Taking notice of (\ref{e_Stability4}), we have
\begin{equation}\label{e_Stability8}
\aligned
&\lambda\frac{(R^{n+1})^2-(R^{n})^2}{\Delta t}
+\lambda\frac{\|\textbf{D}Z^{n+1}\|^2_{l^2}-\|\textbf{D}Z^n\|_{l^2}^2}{2\Delta t}\\
&+\frac{\|\textbf{U}^{n+1}\|^2_{l^2}-\|\textbf{U}^{n}\|_{l^2}^2}{2\Delta t}
+\nu\|D\textbf{U}\|^2
=-M\|\textbf{D}W^{n+1/2}\|_{l^2}^2\leq 0,
\endaligned
\end{equation} 
which implies the desired result.   
\end{proof}

 \section{Error estimates} 
 In this section we carry out an error analysis for the full discrete  scheme \eqref{e_model_r_full_discrete1}-\eqref{e_model_r_full_discrete6} with $\gamma=0$, i.e. for the  Cahn-Hilliard-Stokes system. The analysis for the case of $\gamma=1$, i.e. for the Cahn-Hilliard-Navier-Stokes system,  will be extremely technical as it requires 
 a high order upwind method to deal with the nonlinear convection term. 

\subsection{An auxiliary problem}
 We consider first an auxiliary problem which will be used in the sequel.  

 Let $(\phi,\mu,\textbf{u},p)$ be the solution of Cahn-Hilliard-Stokes system, and set $\textbf{g}=\mu\nabla\phi-\frac{\partial \textbf{u}}{\partial t}$.
 For each time step $n$, we rewrite \eqref{e_modelC}-\eqref{e_modelD} with $\gamma=0$ as
  \begin{subequations}\label{e_auxiliary}
    \begin{align}
   -\nu\Delta\textbf{u}^n+\nabla p^n=\textbf{g}^n
     \quad &\ in\ \Omega\times J,
      \label{e_auxiliaryA}\\
      \nabla\cdot\textbf{u}^n=0
      \quad &\ in\ \Omega\times J,
      \label{_auxiliaryB}
    \end{align}
  \end{subequations}
and consider its approximation by the MAC scheme:
 For each $n=1,\ldots,N$, let $\{\widehat{U}^{n}_{1,i,j+1/2}\}, \{\widehat{U}^{n}_{2,i+1/2,j}\}$ and $\{\widehat{P}^{n}_{i+1/2,j+1/2}\}$ such that
\begin{align}
&-\nu \frac{d_x\widehat{U}^{n+1/2}_{1,i+1/2,j+1/2}-d_x\widehat{U}^{n+1/2}_{1,i-1/2,j+1/2}}{h_i}
-\nu \frac{D_y\widehat{U}^{n+1/2}_{1,i,j+1}-D_y\widehat{U}^{n+1/2}_{1,i,j}}{k_{j+1/2}}\notag\\
&~~~~~+D_x\widehat{P}_{i,j+1/2}^{n+1/2}=g_{1,i,j+1/2}^{n+1/2},\ \ 1\leq i\leq N_x-1,0\leq j\leq N_y-1,\label{e36}\\
&-\nu \frac{D_x\widehat{U}^{n+1/2}_{1,i+1,j}-D_x\widehat{U}^{n+1/2}_{1,i,j}}{h_{i+1/2}}
-\nu \frac{d_y\widehat{U}^{n+1/2}_{2,i+1/2,j+1/2}-d_y\widehat{U}^{n+1/2}_{2,i+1/2,j-1/2}}{k_{j}}\notag\\
&~~~~~+D_y\widehat{P}_{i+1/2,j}^{n+1/2}=g_{2,i+1/2,j}^{n+1/2},\ \ 0\leq i\leq N_x-1,1\leq j\leq N_y-1,\label{e37}\\
&d_x\widehat{U}^{n+1/2}_{1,i+1/2,j+1/2}+d_y\widehat{U}^{n+1/2}_{2,i+1/2,j+1/2}=0,\ \ 0\leq i\leq N_x-1,0\leq j\leq N_y-1,\label{e38}
\end{align}
where the boundary and initial approximations are same as equations \eqref{boundary condition} and \eqref{initial condition}.

Inspired by \cite{dawson1998two}, we extend the work in Rui and Li \cite{rui2017stability} to the above approximation. By following closely the same arguments as in \cite{rui2017stability}, we can prove the following:
\begin{lemma}\label{le_auxiliary} 
Assuming that $\textbf{u}\in W^{3}_{\infty}(J;W^{4}_{\infty}(\Omega))^2$, $p\in W^{3}_{\infty}(J;W^{3}_{\infty}(\Omega))$, we have the following results:
\begin{equation}\label{e127}
\aligned
\|d_x(\widehat{U}^{n+1}_1-{u}^{n+1}_1)\|_{l^2,M}+\|d_y(\widehat{U}^{n+1}_2-{u}^{n+1}_2)\|_{l^2,M}
\leq
O(\Delta t^2+h^2+k^2),
\endaligned
\end{equation}
\begin{equation}\label{e39}
\aligned
\|d_t(\widehat{U}^{n+1}_1-{u}^{n+1}_1)\|_{l^2,T,M}+\|d_t(\widehat{U}^{n+1}_2-{u}^{n+1}_2)\|_{l^2,M,T}
\leq O(\Delta t^2+h^2+k^2),
\endaligned
\end{equation}
\begin{equation}\label{e128}
\aligned
\|\widehat{U}^{n+1}_1-{u}^{n+1}_1\|_{l^2,T,M}+\|\widehat{U}^{n+1}_2-{u}^{n+1}_2\|_{l^2,M,T}
\leq O(\Delta t^2+h^2+k^2),
\endaligned
\end{equation}

\begin{equation}\label{e132}
\aligned
& \|D_y(\widehat{U}^{n+1}_1-{u}^{n+1}_1)\|_{l^2,T_y}\leq O(\Delta t^2+h^2+k^{3/2}),
\endaligned
\end{equation}
\begin{equation}\label{e134}
\aligned
&\|D_x(\widehat{U}^{n+1}_2-{u}^{n+1}_2)\|_{l^2,T_x}\leq O(\Delta t^2+h^{3/2}+k^2),
\endaligned
\end{equation}

\begin{equation}\label{e131}
\aligned
&\left(\sum\limits_{l=1}^{N}\Delta t\|(\widehat{Z}-p)^{l-1/2}\|^2_{l^2,M}\right)^{1/2}\leq O(\Delta t^2+h^2+k^2).
\endaligned
\end{equation}
\end{lemma}
\medskip
\subsection{ discrete LBB condition}
In order to carry out error analysis, we need the discrete LBB condition.

 Here we use the same notation and results as Rui and Li \cite[Lemma 3.3]{rui2017stability}.  
 Let $$b(\textbf{v},q)=-\int_\Omega ~qdiv \textbf{v}dx,~\textbf{v}\in \textbf{V},~q\in W,$$ where
\begin{align*}
&\textbf{V}=H^1_0(\Omega)\times H^1_0(\Omega),
\quad W=\left\{q\in L^2(\Omega): \int_\Omega qdx=0\right\}.
\end{align*}
\begin{figure}[htp]
\centering
\includegraphics[scale=0.26]{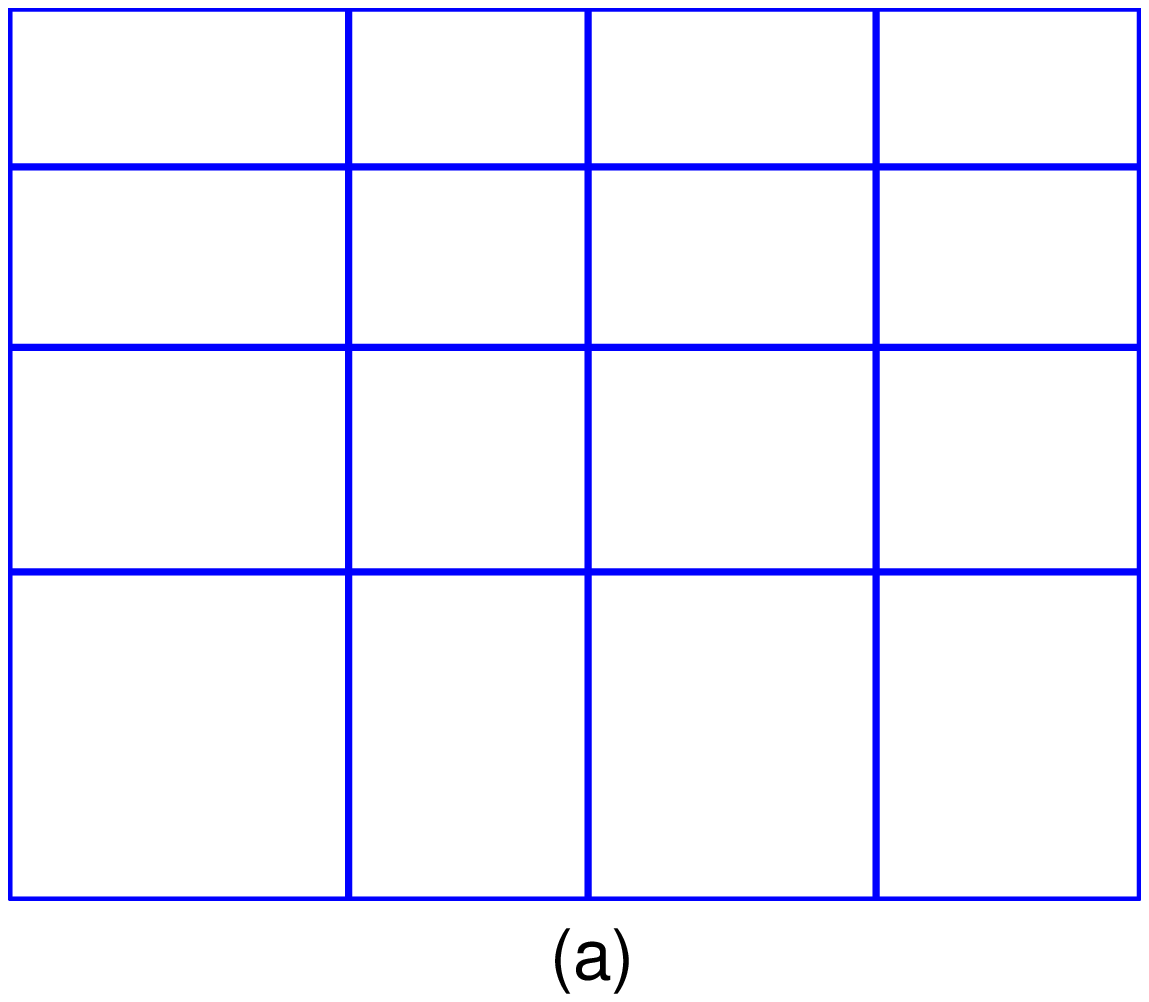}
\includegraphics[scale=0.26]{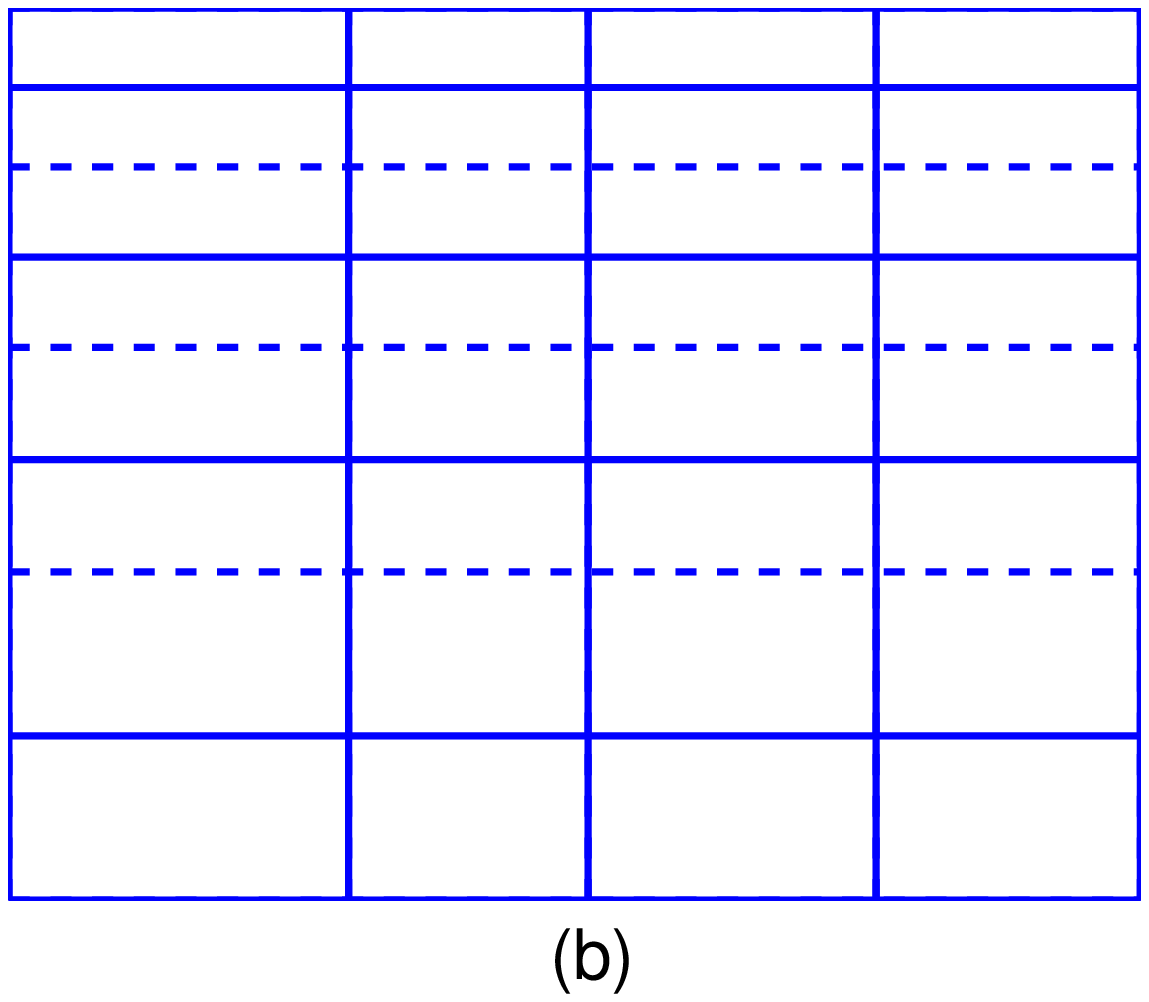}
\includegraphics[scale=0.26]{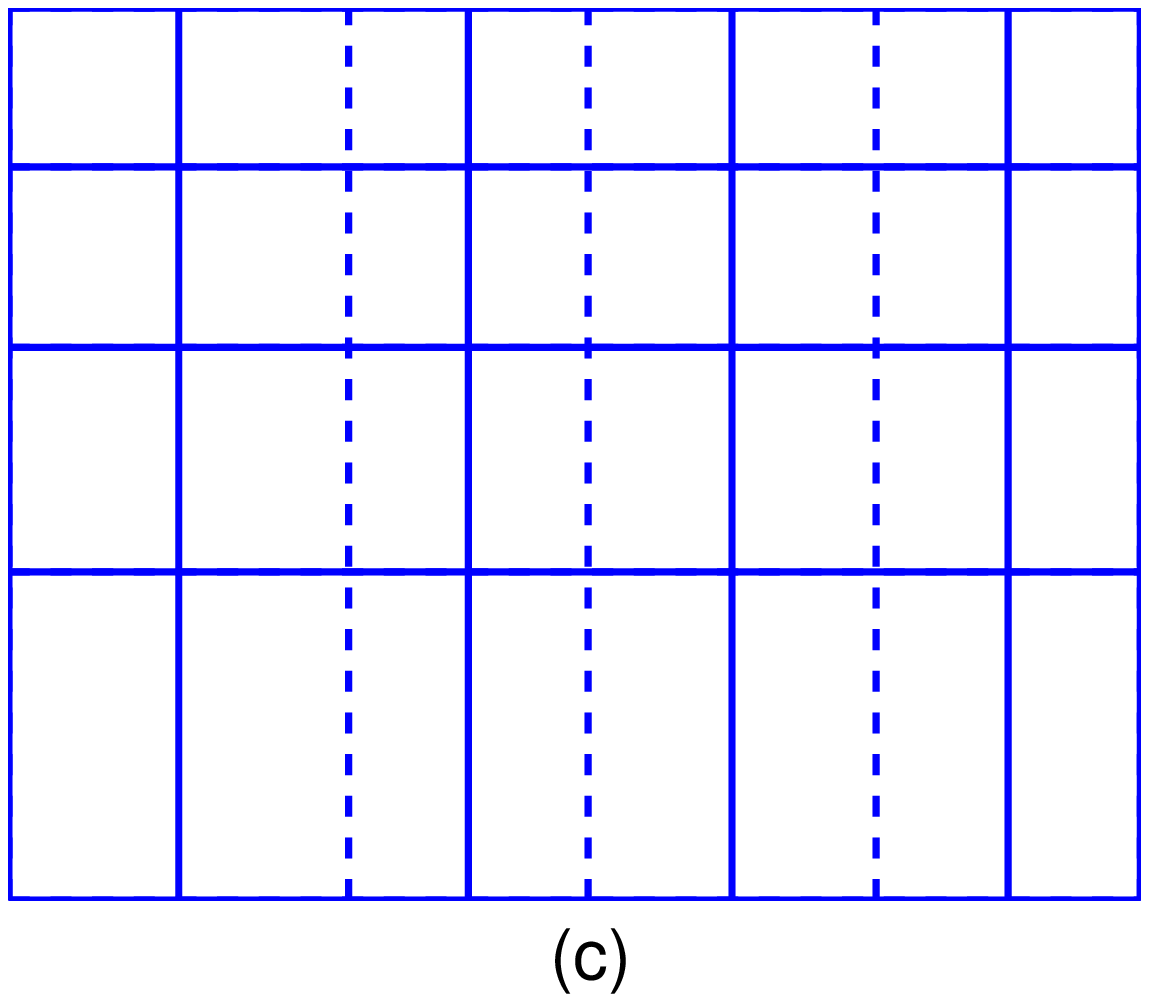}
\caption{Partitions: (a) $\mathcal{T}_h$, (b) $\mathcal{T}_h^1$, (c) $\mathcal{T}_h^2$}\label{fig1}
\end{figure}

Then we construct the finite-dimensional subspaces of $W$ and $\textbf{V}$ by introducing three different partitions $\mathcal{T}_h,\mathcal{T}_h^1,\mathcal{T}_h^2$ of $\Omega$.
The original partition $\delta_x\times \delta_y$ is denoted by $\mathcal{T}_h$ (see Fig \ref{fig1}). The partition $\mathcal{T}_h^1$ is generated by connecting all the midpoints of the vertical sides of $\Omega_{i+1/2,j+1/2}$ and extending the resulting mesh to the boundary $\Gamma$. Similarly, for all $\Omega_{i+1/2,j+1/2}\in \mathcal{T}_h$ we connect all the midpoints of the horizontal sides of $\Omega_{i+1/2,j+1/2}$ and extend
the resulting mesh to the boundary $\Gamma$, then the third partition is obtained which is denoted by $\mathcal{T}_h^2$.

Corresponding to the  quadrangulation $\mathcal{T}_h$,
define $W_h$, a subspace of $W$,
 $$W_h=\left\{q_h:~q_h|_T=constant,~\forall T\in \mathcal{T}_h~and \int_\Omega qdx=0\right\}.$$
  Furthermore, let $\textbf{V}_h$ be a subspace of $\textbf{V}$ such that $\textbf{V}_h$=$S_h^1\times S_h^2$, where
\begin{align*}
&S_h^l=\left\{g\in C^{(0)}(\overline{\Omega}):~g|_{T^l}\in Q_1(T^l),~,\forall T^l\in \mathcal{T}_h^l,~and~g|_\Gamma=0\right\},~l=1,2,\\
\end{align*}
and $Q_1$ denotes the space of all polynomials of degree $\leq 1$ with respect to each of the two variables $x$ and $y$.

Then we introduce the bilinear forms
$$b_h(\textbf{v}_h,q_h)=-\sum_{\Omega_{i+1/2,j+1/2}\in \mathcal{T}_h}\int_{\Omega_{i+1/2,j+1/2}} q_h \Pi_h(div\textbf{v}_h)dx,~\textbf{v}_h\in \textbf{V}_h,~q_h\in W_h,$$
where
\begin{align*}
\Pi_h:~&C^{(0)}(\overline{\Omega}_{i+1/2,j+1/2})\rightarrow Q_0(\Omega_{i+1/2,j+1/2}), ~such~that\\
&(\Pi_h\varphi)_{i+1/2,j+1/2}=\varphi_{i+1/2,j+1/2},~~\forall ~\Omega_{i+1/2,j+1/2}\in \mathcal{T}_h.
\end{align*}
 Then, we have the following result: 
\begin{lemma}\label{le_LBB}
There is a constant $\beta>0$, independent of $h$ and $k$ such that
\begin{equation}\label{e16}
\sup\limits_{\textbf{v}_h\in\textbf{V}_h}\frac{b_h(\textbf{v}_h,q_h)}{\|D\textbf{v}_h\|}\geq \beta\|q_h\|_{l^2,M}~~\forall q_h\in W_h.
\end{equation}
\end{lemma}

\subsection{A first error estimate with a $L^\infty$ bound assumption}
we shall first  derive an error estimate assuming  that there exists two positive constant $C_*$ and $C^*$ such that 
  \begin{subequations}\label{e_hypotheses}
    \begin{align}
    &\displaystyle \|Z^n\|_{\infty}\leq C_*,
    \label{e_hypothesesA}\\
     &\displaystyle \|\textbf{D}Z^n\|_{\infty}\leq C^*.
      \label{e_hypothesesB}
    \end{align}
  \end{subequations}
  Late we shall verify this assumption using an induction process.
  
We define the operator $\textbf{I}_h:~\textbf{V}\rightarrow \textbf{V}_h,$ such that
\begin{equation}\label{e_H1 projection}
\aligned
(\nabla\cdot \textbf{I}_h\textbf{v},w)=(\nabla\cdot \textbf{v},w) \ \forall w\in W_h,
\endaligned
\end{equation}
with approximation properties \cite{dawson1998two}
\begin{align}
\|\textbf{v}-\textbf{I}_h\textbf{v}\|\leq &C\|\textbf{v}\|_1\hat{h}, \label{e_H1 projection_error1}\\
\|\nabla\cdot(\textbf{v}-\textbf{I}_h\textbf{v})\|\leq &C\|\nabla\cdot\textbf{v}\|_1\hat{h},\label{e_H1 projection_error2}
\end{align}
where $\hat{h}=\max\{h,k\}$.

Besides, by the definition of $\textbf{I}_h\textbf{v}$ and the midpoint rule of integration,
the $L^\infty$ norm of the projection is obtained by
\begin{equation}\label{e_H1 projection_error3}
\|\textbf{v}-\textbf{I}_h\textbf{v}\|_{\infty}\leq C\|\textbf{v}\|_{W_{\infty}^2(\Omega)}\hat{h}.
\end{equation}

Furthermore from Dur{\'a}n \cite{duran1990superconvergence}, we have the following 
 estimates which is necessary for the derivative and analysis of our numerical scheme:
 \begin{equation}\label{e_H1 projection_error4}
\|\textbf{v}-\textbf{I}_h\textbf{v}\|_{l^2}\leq C\hat{h}^2.
\end{equation}

For simplicity, we set
\begin{equation*}
\aligned
&\displaystyle e_{\phi}^n=Z^n-\phi^n,
\ \displaystyle e_{\mu}^{n}=W^{n}-\mu^n,
\ \displaystyle e_{r}^n=R^n-r^n,\\
&\displaystyle e_{\textbf{u}}^n=\textbf{U}^n-\widehat{\textbf{U}}^n+\widehat{\textbf{U}}^n-\textbf{u}^n=\widehat{e}_{\textbf{u}}^n+\widetilde{e}_{\textbf{u}}^n,\\
& \displaystyle e_{p}^n=P^n-\widehat{P}^n+\widehat{P}^n-p^n=
\widehat{e}_{p}^n+\widetilde{e}_{p}^n.
\endaligned
\end{equation*}


\begin{lemma}\label{lem: error_estimates1}
Suppose that the hypotheses (\ref{e_hypotheses}) hold, and $\phi\in W^{3}_{\infty}(J;W^{4}_{\infty}(\Omega)),\mu\in L^{\infty}(J;W^{4}_{\infty}(\Omega))$, $\textbf{u}\in W^{3}_{\infty}(J;W^{4}_{\infty}(\Omega))^2$, $p\in W^{3}_{\infty}(J;W^{3}_{\infty}(\Omega))$, then the approximate errors of discrete phase function and chemical potential satisfy
\begin{equation}\label{e_error_estimate13**}
\aligned
&\|e_{\phi}^{m+1}\|_{l^2,M}^2+\frac{M}{2}\sum\limits_{n=0}^{m}\Delta t\|e_{\mu}^{n+1/2}\|_{l^2,M}^2
+\lambda(e_r^{m+1})^2\\
&+\frac{\lambda}{2}\|\textbf{D}e_{\phi}^{m+1}\|^2_{l^2}+\frac{M}{4}\sum_{n=0}^{m}\Delta t\|\textbf{D}e_{\mu}^{n+1/2}\|^2_{l^2}\\
\leq&C\sum_{n=0}^{m+1}\Delta t\|\textbf{D}e_{\phi}^n\|_{l^2}^2+C\sum_{n=0}^{m}\Delta t\|\widehat{e}_{\textbf{u}}^{n+1/2}\|_{l^2}^2\\
&+C\sum_{n=0}^{m+1}\Delta t\|e_{\phi}^n\|_{l^2,M}^2
+C\sum_{n=0}^{m+1}\Delta t(e_r^{n})^2\\
&+C(\Delta t^4+h^4+k^4),\quad   \ m\leq N,
\endaligned
\end{equation}
where the positive constant $C$ is independent of $h$, $k$ and $\Delta t$.
\end{lemma}

\begin{proof}
Denote
\begin{equation*}
\aligned
&\delta_x(\phi)=D_x\phi-\frac{\partial \phi}{\partial x},~\delta_y(\phi)=D_y\phi-\frac{\partial \phi}{\partial y},\\
&\delta_x(\mu)=D_x\mu-\frac{\partial \mu}{\partial x},~\delta_y(\mu)=D_y\mu-\frac{\partial \mu}{\partial y}.
\endaligned
\end{equation*}    
Subtracting (\ref{e_model_rA}) from (\ref{e_model_r_full_discrete1}), we obtain
\begin{equation}\label{e_error_estimate3}
\aligned
~&[d_te_{\phi}]_{i+1/2,j+1/2}^{n+1}=M[d_x(D_xe_{\mu}+\delta_x(\mu))+d_y(D_ye_{\mu}+\delta_y(\mu))]_{i+1/2,j+1/2}^{n+1/2}\\
&-\mathcal{P}_h^y\mathcal{P}_h^x[U_1D_x\tilde{Z}+U_2D_y\tilde{Z}]^{n+1/2}_{i+1/2,j+1/2}+
\textbf{u}^{n+1/2}_{i+1/2,j+1/2}\cdot\nabla\phi^{n+1/2}_{i+1/2,j+1/2}\\
&+T_{1,i+1/2,j+1/2}^{n+1/2}+T_{2,i+1/2,j+1/2}^{n+1/2},\\
\endaligned
\end{equation}
where 
\begin{equation}\label{e_error_estimate31}
\aligned
T_{1,i+1/2,j+1/2}^{n+1/2}=&\frac{\partial \phi}{\partial t}\big|_{i+1/2,j+1/2}^{n+1/2}-[d_t\phi]_{i+1/2,j+1/2}^{n+1}\\
\leq&C\|\phi\|_{W^{3}_{\infty}(J;L^{\infty}(\Omega))}\Delta t^2,
\endaligned
\end{equation}

\begin{equation}\label{e_error_estimate32}
\aligned
T_{2,i+1/2,j+1/2}^{n+1/2}&=M[d_x\frac{\partial \mu}{\partial x}+d_y\frac{\partial \mu}{\partial y}]_{i+1/2,j+1/2}^{n+1/2}-M\Delta \mu^{n+1/2}_{i+1/2,j+1/2}\\
&\leq CM(h^2+k^2)\|\mu\|_{L^{\infty}(J;W^{4}_{\infty}(\Omega))}.
\endaligned
\end{equation}
Subtracting (\ref{e_model_rB}) from (\ref{e_model_r_full_discrete2}) leads to
\begin{equation}\label{e_error_estimate4}
\aligned
e_{\mu,i+1/2,j+1/2}^{n+1/2}=&-\lambda[d_x(D_xe_{\phi}+\delta_x(\phi))+d_y(D_ye_{\phi}+\delta_y(\phi))]_{i+1/2,j+1/2}^{n+1/2}\\
&+\lambda\frac{R^{n+1/2}}{\sqrt{E_1^h(\tilde{Z}^{n+1/2})+\delta}}F^{\prime}(\tilde{Z}_{i+1/2,j+1/2}^{n+1/2})\\
&-\lambda\frac{r^{n+1/2}}{\sqrt{E_1(\phi^{n+1/2})+\delta}}F^{\prime}(\phi_{i+1/2,j+1/2}^{n+1/2})\\&+\lambda T_{3,i+1/2,j+1/2}^{n+1/2},
\endaligned
\end{equation}
where 
\begin{equation}\label{e_error_estimate41}
\aligned
T_{3,i+1/2,j+1/2}^{n+1/2}&=\Delta \phi^{n+1/2}_{i+1/2,j+1/2}-[d_x\frac{\partial \phi}{\partial x}+d_y\frac{\partial \phi}{\partial y}]_{i+1/2,j+1/2}^{n+1/2}\\
&\leq C(h^2+k^2)\|\phi\|_{L^{\infty}(J;W^{4}_{\infty}(\Omega))}.
\endaligned
\end{equation}
Subtracting equation (\ref{e_model_rC}) from equation (\ref{e_model_r_full_discrete3}) gives that
\begin{equation}\label{e_error_estimate5}
\aligned
d_te_r^{n+1}=&\frac{1}{2\sqrt{E_1^h(\tilde{Z}^{n+1/2})+\delta}}(F^{\prime}(\tilde{Z}^{n+1/2}),
d_tZ^{n+1})_{l^2,M}\\
&-\frac{1}{2\sqrt{E_1(\phi^{n+1/2})+\delta}}\int_{\Omega}F^{\prime}(\phi^{n+1/2})\phi^{n+1/2}_t d\textbf{x}+T_{4}^{n+1/2},
\endaligned
\end{equation}
where 
\begin{equation}\label{e_error_estimate51}
\aligned
T_4^{n+1/2}=r_t^{n+1/2}-d_tr^{n+1}\leq
C\|r\|_{W^{3}_{\infty}(J)}\Delta t^2.
\endaligned
\end{equation}
Multiplying equation (\ref{e_error_estimate3}) by $e_{\mu,i+1/2,j+1/2}^{n+1/2}hk$, and making summation on $i,j$ for $0\leq i\leq N_x-1,\ 0\leq j\leq N_y-1$, we have 
\begin{equation}\label{e_error_estimate6}
\aligned
&(d_te_{\phi}^{n+1},e_{\mu}^{n+1/2})_{l^2,M}\\
=&M\left(d_x(D_xe_{\mu}+\delta_x(\mu))^{n+1/2}+d_y(D_ye_{\mu}+\delta_y(\mu))^{n+1/2},e_{\mu}^{n+1/2}\right)_{l^2,M}\\
&-(\mathcal{P}_h^y\mathcal{P}_h^x[U_1D_x\tilde{Z}+U_2D_y\tilde{Z}]^{n+1/2}-
\textbf{u}^{n+1/2}\cdot\nabla\phi^{n+1/2},e_{\mu}^{n+1/2})_{l^2,M}\\
&+(T_1^{n+1/2}, e_{\mu}^{n+1/2})_{l^2,M}+(T_2^{n+1/2}, e_{\mu}^{n+1/2})_{l^2,M}.
\endaligned
\end{equation} 
Recalling Lemma \ref{lemma:U-P-Relation}, the first term on the right hand side of (\ref{e_error_estimate6}) can be estimated as follows:
\begin{equation}\label{e_error_estimate61}
\aligned
&M\left(d_x(D_xe_{\mu}+\delta_x(\mu))^{n+1/2}+d_y(D_ye_{\mu}+\delta_y(\mu))^{n+1/2},e_{\mu}^{n+1/2}\right)_{l^2,M}\\
=&-M\left((D_xe_{\mu}+\delta_x(\mu))^{n+1/2},D_xe_{\mu}^{n+1/2}\right)_{l^2,T,M}\\
&-M\left((D_ye_{\mu}+\delta_y(\mu))^{n+1/2},D_ye_{\mu}^{n+1/2}\right)_{l^2,M,T}\\
=&-M\|\textbf{D}e_{\mu}^{n+1/2}\|^2_{l^2}-M(\delta_x(\mu)^{n+1/2},D_xe_{\mu}^{n+1/2})_{l^2,T,M}\\
&-M(\delta_y(\mu)^{n+1/2},D_ye_{\mu}^{n+1/2})_{l^2,M,T}.
\endaligned
\end{equation} 
With the aid of Cauchy-Schwarz inequality, the last two terms on the right hand side of (\ref{e_error_estimate61}) can be transformed into:
\begin{equation}\label{e_error_estimate611}
\aligned
&-M(\delta_x(\mu)^{n+1/2},D_xe_{\mu}^{n+1/2})_{l^2,M,T}-M(\delta_y(\mu)^{n+1/2},D_ye_{\mu}^{n+1/2})_{l^2,T,M}\\
\leq&\frac{M}{6}\|\textbf{D}e_{\mu}^{n+1/2}\|_{l^2}^2+C\|\mu\|_{L^{\infty}(J;W^{3}_{\infty}(\Omega))}^2(h^4+k^4).
\endaligned
\end{equation} 
The second term on the right hand side of (\ref{e_error_estimate6}) can be transformed into
\begin{equation}\label{e_error_estimate62}
\aligned
&-(\mathcal{P}_h^y\mathcal{P}_h^x[U_1D_x\tilde{Z}+U_2D_y\tilde{Z}]^{n+1/2}-
\textbf{u}^{n+1/2}\cdot\nabla\phi^{n+1/2},e_{\mu}^{n+1/2})_{l^2,M}\\
=&-(\mathcal{P}_h^y\mathcal{P}_h^x[U_1D_x\tilde{Z}+U_2D_y\tilde{Z}]^{n+1/2}-\mathcal{P}_h^y\mathcal{P}_h^x[\widehat{U}_1D_x\tilde{Z}+\widehat{U}_2D_y\tilde{Z}]^{n+1/2},e_{\mu}^{n+1/2})_{l^2,M}\\
&-(\mathcal{P}_h^y\mathcal{P}_h^x[\widehat{U}_1D_x\tilde{Z}+\widehat{U}_2D_y\tilde{Z}]^{n+1/2}-
\mathcal{P}_h^y\mathcal{P}_h^x[u_1D_x\tilde{Z}+u_2D_y\tilde{Z}]^{n+1/2},e_{\mu}^{n+1/2})_{l^2,M}\\
&-(\mathcal{P}_h^y\mathcal{P}_h^x[u_1D_x\tilde{Z}+u_2D_y\tilde{Z}]^{n+1/2}-\textbf{u}^{n+1/2}\cdot\nabla\phi^{n+1/2},e_{\mu}^{n+1/2})_{l^2,M}.
\endaligned
\end{equation} 
Then taking notice of the definition of interpolations $\mathcal{P}_h^x$ and $\mathcal{P}_h^y$ , the first term on the right hand side of (\ref{e_error_estimate62}) can be bounded by
\begin{equation}\label{e_error_estimate621}
\aligned
&-(\mathcal{P}_h^y\mathcal{P}_h^x[U_1D_x\tilde{Z}+U_2D_y\tilde{Z}]^{n+1/2}-\mathcal{P}_h^y\mathcal{P}_h^x[\widehat{U}_1D_x\tilde{Z}+\widehat{U}_2D_y\tilde{Z}]^{n+1/2},e_{\mu}^{n+1/2})_{l^2,M}\\
\leq&C\|\textbf{D}\tilde{Z}\|_{\infty}^2\|\widehat{e}_{\textbf{u}}^{n+1/2}\|_{l^2}^2+C\|e_{\mu}^{n+1/2}\|_{l^2,M}^2.
\endaligned
\end{equation} 
Similarly noting Lemma \ref{le_auxiliary}, the second term on the right hand side of (\ref{e_error_estimate62}) can be estimated by
\begin{equation}\label{e_error_estimate622}
\aligned
&-(\mathcal{P}_h^y\mathcal{P}_h^x[\widehat{U}_1D_x\tilde{Z}+\widehat{U}_2D_y\tilde{Z}]^{n+1/2}-
\mathcal{P}_h^y\mathcal{P}_h^x[u_1D_x\tilde{Z}+u_2D_y\tilde{Z}]^{n+1/2},e_{\mu}^{n+1/2})_{l^2,M}\\
\leq&C\|\textbf{D}\tilde{Z}\|_{\infty}^2\|\widetilde{e}_{\textbf{u}}^{n+1/2}\|_{l^2}^2+C\|e_{\mu}^{n+1/2}\|_{l^2,M}^2\\
\leq&C\|e_{\mu}^{n+1/2}\|_{l^2,M}^2+C(\Delta t^4+h^4+k^4).
\endaligned
\end{equation}
Supposing that $\phi\in W^{2,\infty}(J;L^{\infty}(\Omega))$, the last term on the right hand side of (\ref{e_error_estimate62}) can be estimated by
\begin{equation}\label{e_error_estimate623}
\aligned
&-(\mathcal{P}_h^y\mathcal{P}_h^x[u_1D_x\tilde{Z}+u_2D_y\tilde{Z}]^{n+1/2}-\textbf{u}^{n+1/2}\cdot\nabla\phi^{n+1/2},e_{\mu}^{n+1/2})_{l^2,M}\\
\leq& C\|e_{\mu}^{n+1/2}\|_{l^2,M}^2+C\|\textbf{D}e_{\phi}^n\|_{l^2,M}^2+C\|\textbf{D}e_{\phi}^{n-1}\|_{l^2,M}^2\\
&+C\|\phi\|^2_{W^{2}_{\infty}(J;L^{\infty}(\Omega))}\Delta t^4.
\endaligned
\end{equation} 
Multiplying equation (\ref{e_error_estimate4}) by $d_te_{\phi,i+1/2,j+1/2}^{n+1}hk$, and making summation on $i,j$ for $0\leq i\leq N_x-1,~0\leq j\leq N_y-1$, we have 
\begin{equation}\label{e_error_estimate7}
\aligned
&(e_{\mu}^{n+1/2},d_te_{\phi}^{n+1})_{l^2,M}\\
=&-\lambda(d_x(D_xe_{\phi}+\delta_x(\phi))^{n+1/2}+d_y(D_ye_{\phi}+\delta_y(\phi))^{n+1/2}, d_te_{\phi}^{n+1})_{l^2,M}\\
&+\lambda(\frac{R^{n+1/2}}{\sqrt{E_1^h(\tilde{Z}^{n+1/2})+\delta}}F^{\prime}(\tilde{Z}^{n+1/2})
-\frac{r^{n+1/2}}{\sqrt{E_1(\phi^{n+1/2})+\delta}}F^{\prime}(\phi^{n+1/2}), d_te_{\phi}^{n+1})_{l^2,M}\\
&+\lambda(T_3^{n+1/2},d_te_{\phi}^{n+1})_{l^2,M}.
\endaligned
\end{equation}
Similar to the estimate of equation (\ref{e_Stability2_1}), the 
first term on the right hand side of equation (\ref{e_error_estimate7}) can be transformed into the following:
\begin{equation}\label{e_error_estimate71}
\aligned
&-\lambda(d_x(D_xe_{\phi}+\delta_x(\phi))^{n+1/2}+d_y(D_ye_{\phi}+\delta_y(\phi))^{n+1/2}, d_te_{\phi}^{n+1})_{l^2,M}\\
=&\lambda(D_xe_{\phi}^{n+1/2},d_tD_xe_{\phi}^{n+1})_{l^2,T,M}+\lambda(D_ye_{\phi}^{n+1/2},d_tD_ye_{\phi}^{n+1})_{l^2,M,T}\\
&+\lambda(\delta_x(\phi)^{n+1/2},d_tD_xe_{\phi}^{n+1/2})_{l^2,T,M}
+\lambda(\delta_y(\phi)^{n+1/2},d_tD_ye_{\phi}^{n+1/2})_{l^2,M,T}\\
=&\lambda\frac{\|\textbf{D}e_{\phi}^{n+1}\|^2_{l^2}-\|\textbf{D}e_{\phi}^n\|_{l^2}^2}{2\Delta t}
+\lambda(\delta_x(\phi)^{n+1/2},d_tD_xe_{\phi}^{n+1/2})_{l^2,T,M}\\
&+\lambda(\delta_y(\phi)^{n+1/2},d_tD_ye_{\phi}^{n+1/2})_{l^2,M,T}.
\endaligned
\end{equation}
The second term on the right hand side of equation (\ref{e_error_estimate7}) can be
rewritten as follows:
\begin{equation}\label{e_error_estimate72}
\aligned
&\lambda(\frac{R^{n+1/2}}{\sqrt{E_1^h(\tilde{Z}^{n+1/2})+\delta}}F^{\prime}(\tilde{Z}^{n+1/2})
-\frac{r^{n+1/2}}{\sqrt{E_1(\phi^{n+1/2})+\delta}}F^{\prime}(\phi^{n+1/2}), d_te_{\phi}^{n+1})_{l^2,M}\\
=&\lambda r^{n+1/2}(\frac{F^{\prime}(\tilde{Z}^{n+1/2})}{\sqrt{E_1^h(\tilde{Z}^{n+1/2})+\delta}}-\frac{F^{\prime}(\tilde{\phi}^{n+1/2})}{\sqrt{E_1^h(\tilde{\phi}^{n+1/2})+\delta}}, d_te_{\phi}^{n+1})_{l^2,M}\\
&+\lambda r^{n+1/2}(\frac{F^{\prime}(\tilde{\phi}^{n+1/2})}{\sqrt{E_1^h(\tilde{\phi}^{n+1/2})+\delta}}-\frac{F^{\prime}(\phi^{n+1/2})}{\sqrt{E_1(\phi^{n+1/2})+\delta}}, d_te_{\phi}^{n+1})_{l^2,M}\\
&+\lambda e_{r}^{n+1/2}(\frac{F^{\prime}(\tilde{Z}^{n+1/2})}{\sqrt{E_1^h(\tilde{Z}^{n+1/2})+\delta}},
d_te_{\phi}^{n+1})_{l^2,M}.
\endaligned
\end{equation}
Taking notice of (\ref{e_error_estimate3}), the first term on the right hand side of (\ref{e_error_estimate72}) can be transformed into the following:
\begin{equation}\label{e_error_estimate721}
\aligned
&\lambda r^{n+1/2}(\frac{F^{\prime}(\tilde{Z}^{n+1/2})}{\sqrt{E_1^h(\tilde{Z}^{n+1/2})+\delta}}-\frac{F^{\prime}(\tilde{\phi}^{n+1/2})}{\sqrt{E_1^h(\tilde{\phi}^{n+1/2})+\delta}}, d_te_{\phi}^{n+1})_{l^2,M}\\
=&M\lambda r^{n+1/2}(\frac{F^{\prime}(\tilde{Z}^{n+1/2})}{\sqrt{E_1^h(\tilde{Z}^{n+1/2})+\delta}}-\frac{F^{\prime}(\tilde{\phi}^{n+1/2})}{\sqrt{E_1^h(\tilde{\phi}^{n+1/2})+\delta}}, d_x(D_xe_{\mu}+\delta_x(\mu))^{n+1/2})_{l^2,M}\\
&+M\lambda r^{n+1/2}(\frac{F^{\prime}(\tilde{Z}^{n+1/2})}{\sqrt{E_1^h(\tilde{Z}^{n+1/2})+\delta}}-\frac{F^{\prime}(\tilde{\phi}^{n+1/2})}{\sqrt{E_1^h(\tilde{\phi}^{n+1/2})+\delta}}, d_y(D_ye_{\mu}+\delta_y(\mu))^{n+1/2})_{l^2,M}\\
&-\lambda r^{n+1/2}(\frac{F^{\prime}(\tilde{Z}^{n+1/2})}{\sqrt{E_1^h(\tilde{Z}^{n+1/2})+\delta}}-\frac{F^{\prime}(\tilde{\phi}^{n+1/2})}{\sqrt{E_1^h(\tilde{\phi}^{n+1/2})+\delta}}, \mathcal{P}_h[U_1D_x\tilde{Z}+U_2D_y\tilde{Z}]^{n+1/2}\\
&-\textbf{u}^{n+1/2}_{i+1/2,j+1/2}\cdot\nabla\phi^{n+1/2})_{l^2,M}\\
&+\lambda r^{n+1/2}(\frac{F^{\prime}(\tilde{Z}^{n+1/2})}{\sqrt{E_1^h(\tilde{Z}^{n+1/2})+\delta}}-\frac{F^{\prime}(\tilde{\phi}^{n+1/2})}{\sqrt{E_1^h(\tilde{\phi}^{n+1/2})+\delta}}, T_1^{n+1/2}+T_2^{n+1/2})_{l^2,M}.
\endaligned
\end{equation}
Similar to the estimates in \cite{li2018energy}, and using the Cauchy-Schwartz inequality, we can deduce that 
\begin{equation}\label{e_error_estimate7211}
\aligned
&M\lambda r^{n+1/2}(\frac{F^{\prime}(\tilde{Z}^{n+1/2})}{\sqrt{E_1^h(\tilde{Z}^{n+1/2})+\delta}}-\frac{F^{\prime}(\tilde{\phi}^{n+1/2})}{\sqrt{E_1^h(\tilde{\phi}^{n+1/2})+\delta}}, d_x(D_xe_{\mu}+\delta_x(\mu))^{n+1/2})_{l^2,M}\\
=&-M\lambda r^{n+1/2}(\frac{D_xF^{\prime}(\tilde{Z}^{n+1/2})}{\sqrt{E_1^h(\tilde{Z}^{n+1/2})+\delta}}-\frac{D_xF^{\prime}(\tilde{\phi}^{n+1/2})}{\sqrt{E_1^h(\tilde{\phi}^{n+1/2})+\delta}}, (D_xe_{\mu}+\delta_x(\mu))^{n+1/2})_{l^2,M}\\
\leq& \frac{M}{6}\|D_xe_{\mu}^{n+1/2}\|_{l^2,T,M}^2+C\|r\|^2_{L^{\infty}(J)}(\|e_{\phi}^{n}\|_m^2+\|e_{\phi}^{n-1}\|_{l^2,M}^2)\\
&+C\|r\|^2_{L^{\infty}(J)}(\|D_xe_{\phi}^{n}\|_{l^2,T,M}^2+\|D_xe_{\phi}^{n-1}\|_{l^2,T,M}^2)\\
&+C\|\mu\|_{L^{\infty}(J;W^{3}_{\infty}(\Omega))}^2(h^4+k^4).
\endaligned
\end{equation}
Similarly we can obtain
\begin{equation}\label{e_error_estimate7212}
\aligned
&M\lambda r^{n+1/2}(\frac{F^{\prime}(\tilde{Z}^{n+1/2})}{\sqrt{E_1^h(\tilde{Z}^{n+1/2})+\delta}}-\frac{F^{\prime}(\tilde{\phi}^{n+1/2})}{\sqrt{E_1^h(\tilde{\phi}^{n+1/2})+\delta}}, d_y(D_ye_{\mu}+\delta_y(\mu))^{n+1/2})_{l^2,M}\\
\leq& \frac{M}{6}\|D_ye_{\mu}^{n+1/2}\|_{l^2,M,T}^2+C\|r\|^2_{L^{\infty}(J)}(\|e_{\phi}^{n}\|_{l^2,M}^2+\|e_{\phi}^{n-1}\|_{l^2,M}^2)\\
&+C\|r\|^2_{L^{\infty}(J)}(\|D_ye_{\phi}^{n}\|_{l^2,M,T}^2+\|D_ye_{\phi}^{n-1}\|_{l^2,M,T}^2)\\
&+C\|\mu\|_{L^{\infty}(J;W^{3}_{\infty}(\Omega))}^2(h^4+k^4).
\endaligned
\end{equation}
Then equation (\ref{e_error_estimate721}) can be estimated by:
\begin{equation}\label{e_error_estimate7213}
\aligned
&\lambda r^{n+1/2}(\frac{F^{\prime}(\tilde{Z}^{n+1/2})}{\sqrt{E_1^h(\tilde{Z}^{n+1/2})+\delta}}-\frac{F^{\prime}(\tilde{\phi}^{n+1/2})}{\sqrt{E_1^h(\tilde{\phi}^{n+1/2})+\delta}}, d_te_{\phi}^{n+1})_{l^2,M}\\
\leq &\frac{M}{6}\|\textbf{D}e_{\mu}^{n+1/2}\|_{l^2}^2+C\|r\|_{L^{\infty}(J)}(\|e_{\phi}^{n}\|_{l^2,M}^2+\|e_{\phi}^{n-1}\|_{l^2,M^2})\\
&+C\|r\|_{L^{\infty}(J)}(\|\textbf{D}e_{\phi}^{n}\|_{l^2}^2+\|\textbf{D}e_{\phi}^{n-1}\|_{l^2}^2)
+C\|\textbf{D}\tilde{Z}\|_{\infty}^2\|\widehat{e}_{\textbf{u}}^{n+1/2}\|_{l^2}^2\\
&+C\|\mu\|_{L^{\infty}(J;W^{4}_{\infty}(\Omega))}^2(h^4+k^4)+C\|\phi\|_{W^{3}_{\infty}(J;L^{\infty}(\Omega))}^2\Delta t^4.
\endaligned
\end{equation}
Similar to the estimates of (\ref{e_error_estimate721}),
the second term on the right hand side of (\ref{e_error_estimate72}) can be controlled by:
\begin{equation}\label{e_error_estimate722}
\aligned
&\lambda r^{n+1/2}(\frac{F^{\prime}(\tilde{\phi}^{n+1/2})}{\sqrt{E_1^h(\tilde{\phi}^{n+1/2})+\delta}}-\frac{F^{\prime}(\phi^{n+1/2})}{\sqrt{E_1(\phi^{n+1/2})+\delta}}, d_te_{\phi}^{n+1})_{l^2,M}\\
\leq &\frac{M}{6}\|\textbf{D}e_{\mu}^{n+1/2}\|_{l^2}^2+C\|\textbf{D}e_{\phi}^n\|_{l^2,M}^2+C\|\textbf{D}e_{\phi}^{n-1}\|_{l^2,M}^2\\
&+C\|\textbf{D}\tilde{Z}\|_{\infty}^2\|\widehat{e}_{\textbf{u}}^{n+1/2}\|_{l^2}^2+C\|\phi\|_{W^{3}_{\infty}(J;W^{1}_{\infty}(\Omega))}^2\Delta t^4\\
&+C(\|\mu\|_{L^{\infty}(J;W^{4}_{\infty}(\Omega))}^2+\|\phi\|_{L^{\infty}(J;W^{2}_{\infty}(\Omega))}^2 )(h^4+k^4).
\endaligned
\end{equation}
Multiplying equation (\ref{e_error_estimate5}) by $\lambda (e_{r}^{n+1}+e_{r}^{n})$ leads to 
\begin{equation}\label{e_error_estimate8}
\aligned
\lambda\frac{(e_r^{n+1})^2-(e_r^{n})^2}{\Delta t}=&
\lambda\frac{e_r^{n+1/2}}{\sqrt{E_1^h(\tilde{Z}^{n+1/2})+\delta}}(F^{\prime}(\tilde{Z}^{n+1/2}),
d_tZ^{n+1})_{l^2,M}\\
&-\lambda\frac{e_r^{n+1/2}}{\sqrt{E_1(\phi^{n+1/2})+\delta}}\int_{\Omega}F^{\prime}(\phi^{n+1/2})\phi^{n+1/2}_t d\textbf{x}\\
&+\lambda T_{4}^{n+1/2}\cdot (e_{r}^{n+1}+e_{r}^{n}).
\endaligned
\end{equation}
Then similar to the estimates in \cite{li2018energy}, we have
 \begin{equation}\label{e_error_estimate81}
\aligned
&\lambda\frac{(e_r^{n+1})^2-(e_r^{n})^2}{\Delta t}\\
\leq &\lambda\frac{e_{r}^{n+1/2}}{\sqrt{E_1^h(\tilde{Z}^{n+1/2})+\delta}}(F^{\prime}(\tilde{Z}^{n+1/2}),
d_te_{\phi}^{n+1})_{l^2,M}+\lambda T_{4}^{n+1/2}\cdot (e_{r}^{n+1}+e_{r}^{n})\\
&+C(e_r^{n+1/2})^2+C\|\phi\|^2_{W^{1}_{\infty}(J;L^{\infty}(\Omega))}(\|e_{\phi}^{n}\|_{l^2,M}^2+\|e_{\phi}^{n-1}\|_{l^2,M}^2)\\
&+C\|\phi\|^2_{W^{1}_{\infty}(J;W^{2}_{\infty}(\Omega))}(h^4+k^4).
\endaligned
\end{equation}
Combining the above equations and using Cauchy-Schwarz inequality lead to
\begin{equation}\label{e_error_estimate9}
\aligned
&\lambda\frac{(e_r^{n+1})^2-(e_r^{n})^2}{\Delta t}+\lambda\frac{\|\textbf{D}e_{\phi}^{n+1}\|^2_{l^2}-\|\textbf{D}e_{\phi}^n\|_{l^2}^2}{2\Delta t}+M\|\textbf{D}e_{\mu}^{n+1/2}\|^2_{l^2}\\
\leq &\frac{M}{2}\|\textbf{D}e_{\mu}^{n+1/2}\|_{l^2}^2+C\|e_{\mu}^{n+1/2}\|_{l^2,M}^2+C\|r\|^2_{L^{\infty}(J)}(\|e_{\phi}^{n}\|_{l^2,M}^2+\|e_{\phi}^{n-1}\|_{l^2,M}^2)\\
&+C\|\textbf{D}\tilde{Z}\|_{\infty}^2\|\widehat{e}_{\textbf{u}}^{n+1/2}\|_{l^2}^2+C\|r\|^2_{L^{\infty}(J)}(\|\textbf{D}e_{\phi}^{n}\|_{l^2}^2+\|\textbf{D}e_{\phi}^{n-1}\|_{l^2}^2)
\\
&-\lambda(\delta_x(\phi)^{n+1/2},d_tD_xe_{\phi}^{n+1/2})_{l^2,T,M}
-\lambda(\delta_y(\phi)^{n+1/2},d_tD_ye_{\phi}^{n+1/2})_{l^2,M,T}\\
&+\lambda(T_3^{n+1/2},d_te_{\phi}^{n+1})_{l^2,M}+\lambda T_{4}^{n+1/2}\cdot (e_{r}^{n+1}+e_{r}^{n})\\
&+C(e_r^{n+1/2})^2+C\|\phi\|^2_{W^{1}_{\infty}(J;L^{\infty}(\Omega))}(\|e_{\phi}^{n}\|_{l^2,M}^2+\|e_{\phi}^{n-1}\|_{l^2,M}^2)\\
&+C(\|\phi\|^2_{W^{1}_{\infty}(J;W^{2}_{\infty}(\Omega))}+\|\mu\|_{L^{\infty}(J;W^{4}_{\infty}(\Omega))}^2 )(h^4+k^4)\\
&+C\|\phi\|_{W^{3}_{\infty}(J;W^{1}_{\infty}(\Omega))}^2\Delta t^4.
\endaligned
\end{equation}
Taking notice of that
\begin{equation}\label{e_summation-by-parts formulae in time}
\aligned
\sum_{n=0}^{k}\Delta t&(f^n,d_tg^{n+1})
=-\sum_{n=1}^{k}\Delta t(d_tf^n,g^n)\\
&+(f^k,g^{k+1})+(f^0,g^0).
\endaligned
\end{equation}
Using the above equation and multiplying equation (\ref{e_error_estimate9}) by $\Delta t$, summing over $n$ from $1$ to $m$ result in
\begin{equation}\label{e_error_estimate10}
\aligned
&\lambda(e_r^{m+1})^2+\frac{\lambda}{2}\|\textbf{D}e_{\phi}^{m+1}\|^2_{l^2}+\frac{M}{2}\sum_{n=0}^{m}\Delta t\|\textbf{D}e_{\mu}^{n+1/2}\|^2_{l^2}\\
\leq&C\sum_{n=0}^{m+1}\Delta t\|\textbf{D}e_{\phi}^n\|_{l^2}^2
+\frac{M}{2}\sum_{n=0}^{k+1}\Delta t\|e_{\mu}^{n+1/2}\|_{l^2,M}^2\\
&+C\sum_{n=0}^{m+1}\Delta t\|\widehat{e}_{\textbf{u}}^{n+1/2}\|_{l^2}^2
+C\sum_{n=0}^{m+1}\Delta t\|e_{\phi}^n\|_{l^2,M}^2\\
&+C\sum_{n=0}^{m+1}\Delta t(e_r^{n})^2+C\|\phi\|_{W^{3}_{\infty}(J;W^{1,\infty}(\Omega))}^2\Delta t^4\\
&+C(\|\phi\|^2_{W^{1}_{\infty}(J;W^{4}_{\infty}(\Omega))}+\|\mu\|_{L^{\infty}(J;W^{4}_{\infty}(\Omega))}^2 )(h^4+k^4).
\endaligned
\end{equation}

To proceed to the following the error estimate, we should consider the second term on the right hand side of (\ref{e_error_estimate10}). Multiplying (\ref{e_error_estimate3}) by $e_{\phi,i+1/2,j+1/2}^{n+1/2}hk$, and making summation on $i,j$ for $0\leq i\leq N_x-1,\ 0\leq j\leq N_y-1$, we have 
\begin{equation}\label{e_error_estimate11}
\aligned
&(d_te_{\phi}^{n+1},e_{\phi}^{n+1/2})_{l^2,M}\\
=&M\left(d_x(D_xe_{\mu}+\delta_x(\mu))^{n+1/2}+d_y(D_ye_{\mu}+\delta_y(\mu))^{n+1/2},e_{\phi}^{n+1/2}\right)_{l^2,M}\\
&-(\mathcal{P}_h[U_1D_x\tilde{Z}+U_2D_y\tilde{Z}]^{n+1/2}-
\textbf{u}^{n+1/2}\cdot\nabla\phi^{n+1/2},e_{\phi}^{n+1/2})_{l^2,M}\\
&+(T_1^{n+1/2}, e_{\phi}^{n+1/2})_{l^2,M}+(T_2^{n+1/2}, e_{\phi}^{n+1/2})_{l^2,M}.
\endaligned
\end{equation} 
The first term on the right hand side of (\ref{e_error_estimate11}) can be bounded by
\begin{equation}\label{e_error_estimate111}
\aligned
&M\left(d_x(D_xe_{\mu}+\delta_x(\mu))^{n+1/2}+d_y(D_ye_{\mu}+\delta_y(\mu))^{n+1/2},e_{\phi}^{n+1/2}\right)_{l^2,M}\\
=&-M\left((D_xe_{\mu}+\delta_x(\mu))^{n+1/2}, D_xe_{\phi}^{n+1/2}\right)_{l^2,T,M}\\
&-M\left((D_ye_{\mu}+\delta_y(\mu))^{n+1/2}, D_ye_{\phi}^{n+1/2}\right)_{l^2,M,T}\\
\leq& M\left(e_{\mu}^{n+1/2},d_x(D_xe_{\phi}+\delta_x(\phi))^{n+1/2}+d_y(D_ye_{\phi}+\delta_y(\phi))^{n+1/2} \right)_{l^2,M}\\
&+\frac{M}{4}\|\textbf{D}e_{\mu}^{n+1/2}\|_{l^2}^2+C\|\textbf{D}e_{\phi}^{n+1/2}\|_{l^2}^2\\
&+C(\|\mu\|_{L^{\infty}(J;W^{3}_{\infty}(\Omega))}^2+\|\phi\|_{L^{\infty}(J;W^{3}_{\infty}(\Omega))}^2)(h^4+k^4)\\
\leq&-\frac{M}{2}\|e_{\mu}^{n+1/2}\|_{l^2,M}^2+C(e_r^{n+1}+e_r^{n})^2+
C(\|e_{\phi}^n\|_{l^2,M}^2+\|e_{\phi}^{n-1}\|_{l^2,M}^2)\\
&+\frac{M}{4}\|\textbf{D}e_{\mu}^{n+1/2}\|_{l^2}^2+C\|\textbf{D}e_{\phi}^{n+1/2}\|_{l^2}^2+C\|\phi\|_{L^{\infty}(J;W^{4}_{\infty}(\Omega))}^2(h^4+k^4)\\
&+C(\|\mu\|_{L^{\infty}(J;W^{3}_{\infty}(\Omega))}^2+\|\phi\|_{L^{\infty}(J;W^{3}_{\infty}(\Omega))}^2)(h^4+k^4).
\endaligned
\end{equation}
The second term on the right hand side of (\ref{e_error_estimate11}) can be estimated by
\begin{equation}\label{e_error_estimate112}
\aligned
&-(\mathcal{P}_h[U_1D_x\tilde{Z}+U_2D_y\tilde{Z}]^{n+1/2}-
\textbf{u}^{n+1/2}\cdot\nabla\phi^{n+1/2},e_{\phi}^{n+1/2})_{l^2,M}\\
\leq&C\|\textbf{D}\tilde{Z}\|_{\infty}^2\|\widehat{e}_{\textbf{u}}^{n+1/2}\|_{l^2}^2+C\|\textbf{D}e_{\phi}^n\|_{l^2,M}^2+C\|\textbf{D}e_{\phi}^{n-1}\|_{l^2,M}^2\\
&+C\|e_{\phi}^{n+1/2}\|_{l^2,M}^2+C(\Delta t^4+h^4+k^4).
\endaligned
\end{equation} 
Combining (\ref{e_error_estimate11}) with (\ref{e_error_estimate111}) and (\ref{e_error_estimate112}), multiplying by $2\Delta t$, and summing over $n$ from $1$ to $m$ give that
\begin{equation}\label{e_error_estimate12}
\aligned
&\|e_{\phi}^{m+1}\|_{l^2,M}^2+M\sum\limits_{n=0}^{m}\Delta t\|e_{\mu}^{n+1/2}\|_{l^2,M}^2\\
\leq& C\sum\limits_{n=0}^{m}\Delta t(e_r^{n+1})^2+C\sum\limits_{n=0}^{m}\Delta t\|e_{\phi}^{n+1}\|_{l^2,M}^2+C\sum\limits_{n=0}^{m}\Delta t\|\widehat{e}_{\textbf{u}}^{n+1/2}\|_{l^2}^2\\
&+\frac{M}{4}\sum\limits_{n=0}^{k}\Delta t\|\textbf{D}e_{\mu}^{n+1/2}\|_{l^2}^2
+C\sum\limits_{n=0}^{k}\Delta t\|\textbf{D}e_{\phi}^{n+1/2}\|_{l^2}^2\\
&+C(\|\mu\|_{L^{\infty}(J;W^{4}_{\infty}(\Omega))}^2+\|\phi\|_{L^{\infty}(J;W^{4}_{\infty}(\Omega))}^2)(h^4+k^4)\\
&+C\|\phi\|_{W^{3}_{\infty}(J;L^{\infty}(\Omega))}^2\Delta t^4.
\endaligned
\end{equation} 
Combining (\ref{e_error_estimate10}) with the above equation leads to
\begin{equation}\label{e_error_estimate13}
\aligned
&\|e_{\phi}^{m+1}\|_{l^2,M}^2+\frac{M}{2}\sum\limits_{n=0}^{m}\Delta t\|e_{\mu}^{n+1/2}\|_{l^2,M}^2
+\lambda(e_r^{m+1})^2\\
&+\frac{\lambda}{2}\|\textbf{D}e_{\phi}^{m+1}\|^2_{l^2}+\frac{M}{4}\sum_{n=0}^{m}\Delta t\|\textbf{D}e_{\mu}^{n+1/2}\|^2_{l^2}\\
\leq&C\sum_{n=0}^{m+1}\Delta t\|\textbf{D}e_{\phi}^n\|_{l^2}^2+C\sum_{n=0}^{m}\Delta t\|\widehat{e}_{\textbf{u}}^{n+1/2}\|_{l^2}^2\\
&+C\sum_{n=0}^{m+1}\Delta t\|e_{\phi}^n\|_{l^2,M}^2
+C\sum_{n=0}^{m+1}\Delta t(e_r^{n})^2\\
&+C(\Delta t^4+h^4+k^4).
\endaligned
\end{equation}
\end{proof}
\medskip

\begin{lemma}\label{lem: error_estimates2}
Suppose that the hypotheses (\ref{e_hypotheses}) hold, and $\phi\in W^{3}_{\infty}(J;W^{4}_{\infty}(\Omega)),\mu\in L^{\infty}(J;W^{4}_{\infty}(\Omega))$, $\textbf{u}\in W^{3}_{\infty}(J;W^{4}_{\infty}(\Omega))^2$, $p\in W^{3}_{\infty}(J;W^{3}_{\infty}(\Omega))$, then for the case of Stokes equation, the approximate errors of discrete velocity and pressure satisfy
\begin{equation}\label{e_error_estimate22*****}
\aligned
&\|\widehat{e}_{\textbf{u}}^{m+1}\|_{l^2}^2
+\|\textbf{D}\widehat{e}_{\textbf{u}}^{m+1}\|^2
+\sum\limits_{n=0}^{m}\Delta t\|\widehat{e}_{p}^{n+1/2}\|_{l^2,M}^2\\
\leq &C\sum\limits_{n=0}^{m}\Delta t\|e_{\mu}^{n+1/2}\|_{l^2,M}^2+C\sum\limits_{n=0}^{m}\Delta t\|e_{\phi}^{n}\|_{l^2,M}^2\\
&+C(\Delta t^4+h^4+k^4),\quad   \ m\leq N,
\endaligned
\end{equation}
where the positive constant $C$ is independent of $h$, $k$ and $\Delta t$.
\end{lemma}

\begin{proof}
Subtracting (\ref{e36}) from (\ref{e_model_r_full_discrete4}) for the case of Stokes equation with $\gamma=0$, we can obtain
\begin{equation}\label{e_error_estimate14}
\aligned
&d_{t}\widehat{e}_{\textbf{u},1,i,j+1/2}^{n+1}-\nu \frac{d_x\widehat{e}_{\textbf{u},1,i+1/2,j+1/2}^{n+1/2}-d_x\widehat{e}_{\textbf{u},1,i-1/2,j+1/2}^{n+1/2}}{h_i}\\
&-\nu \frac{D_y\widehat{e}_{\textbf{u},1,i,j+1}^{n+1/2}-D_y\widehat{e}_{\textbf{u},1,i,j}^{n+1/2}}{k_{j+1/2}}
+D_x\widehat{e}_{p,i,j+1/2}^{n+1/2}\\
=&\mathcal{P}_h W_{i,j+1/2}^{n+1/2}[D_x\tilde{Z}]_{i,j+1/2}^{n+1/2}-\mu_{i,j+1/2}^{n+1/2}\frac{\partial \phi}{\partial x}_{i,j+1/2}^{n+1/2}\\
&+\frac{\partial u_1}{\partial t}|_{i,j+1/2}^{n+1/2}-[d_t\widehat{U}_1]_{i,j+1/2}^{n+1}.
\endaligned
\end{equation}
For a discrete function $\{v^{n}_{1,i,j+1/2}\}$ such that $v^{n}_{1,i,j+1/2}|_{\partial \Omega}=0$,  multiplying (\ref{e_error_estimate14}) by times $v^{n}_{1,i,j+1/2}hk$
and make summation for $i,j$ with $i=1,\cdots,N_x-1,~j=0,\cdots,N_y-1$, and recalling Lemma \ref{lemma:U-P-Relation} lead to 
\begin{equation}\label{e_error_estimate15}
\aligned
&(d_{t}\widehat{e}_{\textbf{u},1}^{n+1},v^{n}_1)_{l^2,T,M}+\nu (d_x\widehat{e}_{\textbf{u},1}^{n+1/2},d_xv^{n}_1)_{l^2,M}\\
&+\nu (D_y\widehat{e}_{\textbf{u},1}^{n+1/2},D_yv^{n}_1)_{l^2,T_y}
-(\widehat{e}_{p}^{n+1/2},d_xv^{n}_1)_{l^2,M}\\
=&(\mathcal{P}_h W^{n+1/2}[D_x\tilde{Z}]^{n+1/2}-\mu^{n+1/2}\frac{\partial \phi^{n+1/2}}{\partial x},v^{n}_1)_{l^2,T,M}\\
&+(\frac{\partial u_1^{n+1/2}}{\partial t}-d_t\widehat{U}_1^{n+1},v^{n}_1)_{l^2,T,M}.
\endaligned
\end{equation}
Similarly in the $y$ direction, we have 
\begin{equation}\label{e_error_estimate16}
\aligned
&(d_{t}\widehat{e}_{\textbf{u},2}^{n+1},v^{n}_2)_{l^2,M,T}+\nu (d_y\widehat{e}_{\textbf{u},2}^{n+1/2},d_yv^{n}_2)_{l^2,M}\\
&+\nu (D_x\widehat{e}_{\textbf{u},2}^{n+1/2},D_xv^{n}_2)_{l^2,T_x}
-(\widehat{e}_{p}^{n+1/2},d_yv^{n}_2)_{l^2,M}\\
=&(\mathcal{P}_h W^{n+1/2}[D_y\tilde{Z}]^{n+1/2}-\mu^{n+1/2}\frac{\partial \phi^{n+1/2}}{\partial y},v^{n}_2)_{l^2,M,T}\\
&+(\frac{\partial u_2^{n+1/2}}{\partial t}-d_t\widehat{U}_2^{n+1},v^{n}_2)_{l^2,M,T}.
\endaligned
\end{equation}
Adding (\ref{e_error_estimate15}) and (\ref{e_error_estimate16}) results in
\begin{equation}\label{e_error_estimate17}
\aligned
&(d_{t}\widehat{e}_{\textbf{u},1}^{n+1},v^{n}_1)_{l^2,T,M}+(d_{t}\widehat{e}_{\textbf{u},2}^{n+1},v^{n}_2)_{l^2,M,T}+\nu (d_x\widehat{e}_{\textbf{u},1}^{n+1/2},d_xv^{n}_1)_{l^2,M}\\
&+\nu (D_y\widehat{e}_{\textbf{u},1}^{n+1/2},D_yv^{n}_1)_{l^2,T_y}
+\nu (d_y\widehat{e}_{\textbf{u},2}^{n+1/2},d_yv^{n}_2)_{l^2,M}\\
&+\nu (D_x\widehat{e}_{\textbf{u},2}^{n+1/2},D_xv^{n}_2)_{l^2,T_x}
-(\widehat{e}_{p}^{n+1/2},d_xv^{n}_1+d_yv^{n}_2)_{l^2,M}\\
=&(\mathcal{P}_h W^{n+1/2}[D_x\tilde{Z}]^{n+1/2}-\mu^{n+1/2}\frac{\partial \phi^{n+1/2}}{\partial x},v^{n}_1)_{l^2,T,M}\\
&+(\mathcal{P}_h W^{n+1/2}[D_y\tilde{Z}]^{n+1/2}-\mu^{n+1/2}\frac{\partial \phi^{n+1/2}}{\partial y},v^{n}_2)_{l^2,M,T}\\
&+(\frac{\partial u_1^{n+1/2}}{\partial t}-d_t\widehat{U}_1^{n+1},v^{n}_1)_{l^2,T,M}\\
&+(\frac{\partial u_2^{n+1/2}}{\partial t}-d_t\widehat{U}_2^{n+1},v^{n}_2)_{l^2,M,T}.
\endaligned
\end{equation}
Recalling the definition of the interpolation operator $\mathcal{P}_h$ and assuming that (\ref{e_hypothesesB}) holds, the first term on the right hand side of (\ref{e_error_estimate17}) can be transformed into the following:
\begin{equation}\label{e_error_estimate171}
\aligned
&(\mathcal{P}_h W^{n+1/2}[D_x\tilde{Z}]^{n+1/2}-\mu^{n+1/2}\frac{\partial \phi^{n+1/2}}{\partial x},v^{n}_1)_{l^2,T,M}\\
=& ((\mathcal{P}_h W^{n+1/2}-\mathcal{P}_h \mu^{n+1/2})[D_x\tilde{Z}]^{n+1/2},v^{n}_1)_{l^2,T,M}\\
&+((\mathcal{P}_h \mu^{n+1/2}-\mu^{n+1/2})[D_x\tilde{Z}]^{n+1/2},v^{n}_1)_{l^2,T,M}\\
&+(\mu^{n+1/2}([D_x\tilde{Z}]^{n+1/2}-\frac{\partial \phi^{n+1/2}}{\partial x}),v^{n}_1)_{l^2,T,M}\\
\leq&C\|e_{\mu}^{n+1/2}\|_{l^2,M}^2+C\|e_{\phi}^{n}\|_{l^2,M}^2+C\|e_{\phi}^{n-1}\|_{l^2,M}^2\\
&+\frac{1}{4}\|v^{n}_1\|_{l^2,T,M}^2+C(\Delta t^4+h^4+k^4).
\endaligned
\end{equation}
Similarly the second term on the right hand side of (\ref{e_error_estimate17}) can be estimated by
\begin{equation}\label{e_error_estimate172}
\aligned
&(\mathcal{P}_h W^{n+1/2}[D_y\tilde{Z}]^{n+1/2}-\mu^{n+1/2}\frac{\partial \phi^{n+1/2}}{\partial y},v^{n}_2)_{l^2,M,T}\\
\leq&C\|e_{\mu}^{n+1/2}\|_{l^2,M}^2+C\|e_{\phi}^{n}\|_{l^2,M}^2+C\|e_{\phi}^{n-1}\|_{l^2,M}^2\\
&+\frac{1}{4}\|v^{n}_2\|_{l^2,M,T}^2+C(\Delta t^4+h^4+k^4).
\endaligned
\end{equation}
Taking notice of Lemma \ref{le_auxiliary} and using Cauchy-Schwarz inequality, the last two terms on the right hand side of (\ref{e_error_estimate17}) can be controlled by
\begin{equation}\label{e_error_estimate173}
\aligned
&(\frac{\partial u_1^{n+1/2}}{\partial t}-d_t\widehat{U}_1^{n+1},v^{n}_1)_{l^2,T,M}
+(\frac{\partial u_2^{n+1/2}}{\partial t}-d_t\widehat{U}_2^{n+1},v^{n}_2)_{l^2,M,T}\\
\leq &\frac{1}{4}\|\textbf{v}^n\|_{l^2}^2+C(\Delta t^4+h^4+k^4).
\endaligned
\end{equation}
Using Lemma \ref{le_LBB} and the discrete Poincar$\acute{e}$ inequality, we can obtain
\begin{equation}\label{e_error_estimate18}
\aligned
\beta\|\widehat{e}_{p}^{n+1/2}\|_{l^2,M}
\leq&\sup\limits_{\textbf{v}\in \textbf{V}_h}\frac{(\widehat{e}_{p}^{n+1/2},d_xv^{n}_1+d_yv^{n}_2)_{l^2,M}}{\|D\textbf{v}\|}\\
\leq&C(\|d_{t}\widehat{e}_{\textbf{u},1}^n\|_{l^2,T,M}+\|d_{t}\widehat{e}_{\textbf{u},2}^n\|_{l^2,M,T}+\|d_x\widehat{e}_{\textbf{u},1}^{n+1/2}\|_{l^2,M}\\
&+\|D_y\widehat{e}_{\textbf{u},1}^{n+1/2}\|_{l^2,T_y}
+\|d_y\widehat{e}_{\textbf{u},2}^{n+1/2}\|_{l^2,M}
+\|D_x\widehat{e}_{\textbf{u},2}^{n+1/2}\|_{l^2,T_x})\\
&+C\|e_{\mu}^{n+1/2}\|_{l^2,M}+C\|e_{\phi}^{n}\|_{l^2,M}+C\|e_{\phi}^{n-1}\|_{l^2,M}\\
&+O(\Delta t^2+h^2+k^2).
\endaligned
\end{equation}
Setting $v^{n}_{1,i,j+1/2}=d_{t}\widehat{e}_{\textbf{u},1,i,j+1/2}^{n+1}$, $v^{n}_{2,i+1/2,j}=d_{t}\widehat{e}_{\textbf{u},2,i+1/2,j}^{n+1}$ in (\ref{e_error_estimate17}) leads to 
\begin{equation}\label{e_error_estimate19}
\aligned
&\|d_{t}\widehat{e}_{\textbf{u},1}^{n+1}\|_{l^2,T,M}^2+\|d_{t}\widehat{e}_{\textbf{u},2}^{n+1}\|_{l^2,M,T}^2+\nu \frac{\|\textbf{D}\widehat{e}_{\textbf{u}}^{n+1}\|^2-\|\textbf{D}\widehat{e}_{\textbf{u}}^{n}\|^2}{2\Delta t}\\
=&(\mathcal{P}_h W^{n+1/2}[D_x\tilde{Z}]^{n+1/2}-\mu^{n+1/2}\frac{\partial \phi^{n+1/2}}{\partial x},d_{t}\widehat{e}_{\textbf{u},1}^{n+1})_{l^2,T,M}\\
&+(\mathcal{P}_h W^{n+1/2}[D_y\tilde{Z}]^{n+1/2}-\mu^{n+1/2}\frac{\partial \phi^{n+1/2}}{\partial y},d_{t}\widehat{e}_{\textbf{u},2}^{n+1})_{l^2,M,T}\\
&+(\frac{\partial u_1^{n+1/2}}{\partial t}-d_t\widehat{U}_1^{n+1},d_{t}\widehat{e}_{\textbf{u},1}^{n+1})_{l^2,T,M}\\
&+(\frac{\partial u_2^{n+1/2}}{\partial t}-d_t\widehat{U}_2^{n+1},d_{t}\widehat{e}_{\textbf{u},2}^{n+1})_{l^2,M,T}.
\endaligned
\end{equation}
Noting (\ref{e_error_estimate171})-(\ref{e_error_estimate173}), we have
\begin{equation}\label{e_error_estimate20}
\aligned
&\|d_{t}\widehat{e}_{\textbf{u},1}^{n+1}\|_{l^2,T,M}^2+\|d_{t}\widehat{e}_{\textbf{u},2}^{n+1}\|_{l^2,M,T}^2+\nu \frac{\|\textbf{D}\widehat{e}_{\textbf{u}}^{n+1}\|^2-\|\textbf{D}\widehat{e}_{\textbf{u}}^{n}\|^2}{2\Delta t}\\
\leq&C\|e_{\mu}^{n+1/2}\|_{l^2,M}^2+C\|e_{\phi}^{n}\|_{l^2,M}^2+C\|e_{\phi}^{n-1}\|_{l^2,M}^2\\
&+\frac{1}{2}\|d_{t}\widehat{e}_{\textbf{u},1}^{n+1}\|_{l^2,T,M}^2+\frac{1}{2}\|d_{t}\widehat{e}_{\textbf{u},2}^{n+1}\|_{l^2,M,T}^2\\
&+C(\Delta t^4+h^4+k^4).
\endaligned
\end{equation}
Multiplying (\ref{e_error_estimate20}) by $2\Delta t$, and summing over $n$ from $1$ to $m$ result in
\begin{equation}\label{e_error_estimate21}
\aligned
&\sum\limits_{n=0}^{m}\Delta t(\|d_{t}\widehat{e}_{\textbf{u},1}^{n+1}\|_{l^2,T,M}^2+\|d_{t}\widehat{e}_{\textbf{u},2}^{n+1}\|_{l^2,M,T}^2)\\
&+\nu\|\textbf{D}\widehat{e}_{\textbf{u}}^{m+1}\|^2-\nu\|\textbf{D}\widehat{e}_{\textbf{u}}^{0}\|^2\\
\leq& C\sum\limits_{n=0}^{m}\Delta t\|e_{\mu}^{n+1/2}\|_{l^2,M}^2+C\sum\limits_{n=0}^{m}\Delta t\|e_{\phi}^{n}\|_{l^2,M}^2\\
&+C(\Delta t^4+h^4+k^4).
\endaligned
\end{equation}
Since $\widehat{e}_{\textbf{u},1,0,j+1/2}^{n}=\widehat{e}_{\textbf{u},1,N_x,j+1/2}^{n}
$ and $\widehat{e}_{\textbf{u},2,i+1/2,0}^{n}=\widehat{e}_{\textbf{u},2,i+1/2,N_y}^{n}
$, then we can obtain the following discrete Poincar\'{e} inequality.
\begin{equation}\label{e_error_estimate22}
\aligned
&\|\widehat{e}_{\textbf{u}}^{m+1}\|_{l^2}^2
\leq C\|\textbf{D}\widehat{e}_{\textbf{u}}^{m+1}\|^2\\
\leq &C\sum\limits_{n=0}^{m}\Delta t\|e_{\mu}^{n+1/2}\|_{l^2,M}^2+C\sum\limits_{n=0}^{m}\Delta t\|e_{\phi}^{n}\|_{l^2,M}^2\\
&+C(\Delta t^4+h^4+k^4).
\endaligned
\end{equation}
Recalling (\ref{e_error_estimate18}), we have
\begin{equation}\label{e_error_estimate23}
\aligned
\sum\limits_{n=0}^{m}\Delta t\|\widehat{e}_{p}^{n+1/2}\|_{l^2,M}
\leq &C\sum\limits_{n=0}^{m}\Delta t\|e_{\mu}^{n+1/2}\|_{l^2,M}^2+C\sum\limits_{n=0}^{m}\Delta t\|e_{\phi}^{n}\|_{l^2,M}^2\\
&+C(\Delta t^4+h^4+k^4),
\endaligned
\end{equation}
which leads to the desired result (\ref{e_error_estimate22*****}).
\end{proof}

\subsection{Verification of the hypotheses \eqref{e_hypotheses} and the main results}

\begin{lemma}
 Suppose that $\phi\in W^{1}_{\infty}(J;W^{4}_{\infty}(\Omega)) \cap W^{3}_{\infty}(J;W^{1}_{\infty}(\Omega)),\mu\in L^{\infty}(J;W^{4}_{\infty}(\Omega))$, $\textbf{u}\in W^{3}_{\infty}(J;W^{4}_{\infty}(\Omega))^2$, $p\in W^{3}_{\infty}(J;W^{3}_{\infty}(\Omega))$ and $\Delta t\leq C(h+k)$, then the hypotheses \eqref{e_hypotheses} holds.
\end{lemma}

\begin{proof}
 The proof of (\ref{e_hypothesesA}) is essentially identical with the estimates in \cite{li2018energy}. Thus we only provide a detail proof for (\ref{e_hypothesesB}) below.
 
\textit{\textbf{Step 1}} (Definition of $C^*$): Using the scheme (\ref{e_model_r_full_discrete1})-(\ref{e_model_r_full_discrete6}) for $n=0$, Lemma \ref{lem: error_estimates1} and \ref{lem: error_estimates2}, and the inverse assumption, we can get the approximation $\textbf{D}Z^1$ and the following property:
\begin{equation*}
\aligned
\|\textbf{D}Z^1\|_{\infty}=& \|\textbf{D}Z^1-\textbf{I}_h\textbf{D}\phi^1\|_{\infty}+\|\textbf{I}_h\textbf{D}\phi^1-\textbf{D}\phi^1\|_{\infty}+\|\textbf{D}\phi^1\|_{\infty}\\
\leq&C\hat{h}^{-1}\|\textbf{D}Z^1-\textbf{I}_h\textbf{D}\phi^1\|_{l^2}+\|\textbf{I}_h\textbf{D}\phi^1-\textbf{D}\phi^1\|_{\infty}+\|\textbf{D}\phi^1\|_{\infty}\\
\leq& C\hat{h}^{-1}(\|\textbf{D}e_{\phi}^1\|_{l^2}+\|\textbf{I}_h\textbf{D}\phi^1-\textbf{D}\phi^1\|_{l^2})+\|\textbf{I}_h\textbf{D}\phi^1-\textbf{D}\phi^1\|_{\infty}+\|\textbf{D}\phi^1\|_{\infty}\\
\leq& C\hat{h}^{-1}(\Delta t^2+\hat{h}^2)+\|\textbf{D}\phi^1\|_{\infty}\leq C.
\endaligned
\end{equation*}
where $\hat{h}$ and $\Delta t$ are selected such that $\hat{h}^{-1}\Delta t^2$ is sufficiently small.

 Thus define the positive constant $C^*$ independent of $\hat{h}$ and $\Delta t$ such that
\begin{align*}
C^*&\geq \max\{\|\textbf{D}Z^1\|_{\infty}, 2\|\textbf{D}Z^{n}\|_{\infty}\}.
\end{align*}

\textit{\textbf{Step 2}} (Induction): By the definition of $C^*$, it is trivial that hypothesis (\ref{e_hypothesesB}) holds true for $l=1$. Supposing that $\|\textbf{D}Z^{l-1}\|_{\infty}\leq C^*$ holds true for an integer $l=1,\cdots,N-1$, by Lemmas \ref{lem: error_estimates1} and \ref{lem: error_estimates2} with $m=l$, we have that
$$\|\textbf{D}e_{\phi}^l\|_{l^2}\leq C(\hat{h}^2+\Delta t^2).$$
Next we prove that $\|\textbf{D}Z^{l}\|_{\infty}\leq C^*$ holds true.
Since
\begin{equation}\label{e_hypothesis_proof1}
\aligned
\|\textbf{D}Z^l\|_{\infty}=& \|\textbf{D}Z^l-\textbf{I}_h\textbf{D}\phi^l\|_{\infty}+\|\textbf{I}_h\textbf{D}\phi^l-\textbf{D}\phi^l\|_{\infty}+\|\textbf{D}\phi^l\|_{\infty}\\
\leq& C\hat{h}^{-1}(\|\textbf{D}e_{\phi}^l\|_{l^2}+\|\textbf{I}_h\textbf{D}\phi^l-\textbf{D}\phi^l\|_{l^2})+\|\textbf{I}_h\textbf{D}\phi^l-\textbf{D}\phi^l\|_{\infty}+\|\textbf{D}\phi^l\|_{\infty}\\
\leq& C_1\hat{h}^{-1}(\Delta t^2+\hat{h}^2)+\|\textbf{D}\phi^l\|_{\infty}.
\endaligned
\end{equation}
Let $\Delta t\leq C_2\hat{h}$ and a positive constant $\hat{h}_1$ be small enough to satisfy
$$C_1(1+C_2^2)\hat{h}_1\leq\frac{C^*}{2}.$$
Then for $\hat{h}\in (0,\hat{h}_1],$ equation (\ref{e_hypothesis_proof1}) can be bounded by
\begin{equation}\label{e_hypothesis_proof2}
\aligned
\|\textbf{D}Z^l\|_{\infty}\leq& C_1\hat{h}^{-1}(\Delta t^2+\hat{h}^2)+\|\textbf{D}\phi^l\|_{\infty}\\
\leq&C_1(1+C_2^2)\hat{h}_1+\frac{C^*}{2}\leq C^*.
\endaligned
\end{equation}
Then the proof of induction hypothesis (\ref{e_hypothesesB}) ends.
\end{proof}

Recalling (\ref{e_error_estimate22}), we can transform (\ref{e_error_estimate13**}) into the following:
\begin{equation}\label{e_error_estimate13**_transformed}
\aligned
&\|e_{\phi}^{m+1}\|_{l^2,M}^2+\frac{M}{2}\sum\limits_{n=0}^{m}\Delta t\|e_{\mu}^{n+1/2}\|_{l^2,M}^2
+\lambda(e_r^{m+1})^2\\
&+\frac{\lambda}{2}\|\textbf{D}e_{\phi}^{m+1}\|^2_{l^2}+\frac{M}{4}\sum_{n=0}^{m}\Delta t\|\textbf{D}e_{\mu}^{n+1/2}\|^2_{l^2}\\
\leq&C\sum_{n=0}^{m+1}\Delta t\|\textbf{D}e_{\phi}^n\|_{l^2}^2+C\sum_{n=0}^{m}\Delta t\|D\widehat{e}_{\textbf{u}}^{n+1/2}\|^2\\
&+C\sum_{n=0}^{m+1}\Delta t\|e_{\phi}^n\|_{l^2,M}^2
+C\sum_{n=0}^{m+1}\Delta t(e_r^{n})^2\\
&+C(\Delta t^4+h^4+k^4),\quad   \ m\leq N,
\endaligned
\end{equation}
Multiplying (\ref{e_error_estimate13**_transformed}) and (\ref{e_error_estimate22*****}) by $4C$ and $M$ respectively and using Gronwall's inequality, we can deduce that
\begin{equation}\label{e_error_estimate24}
\aligned
&\|e_{\phi}^{m+1}\|_{l^2,M}^2+\sum\limits_{n=0}^{m}\Delta t\|e_{\mu}^{n+1/2}\|_{l^2,M}^2
+(e_r^{m+1})^2\\
&+\|\textbf{D}e_{\phi}^{m+1}\|^2_{l^2}+\sum_{n=0}^{m}\Delta t\|\textbf{D}e_{\mu}^{n+1/2}\|^2_{l^2}+\|\widehat{e}_{\textbf{u}}^{m+1}\|_{l^2}^2\\
&+\|\textbf{D}\widehat{e}_{\textbf{u}}^{m+1}\|^2
+\sum\limits_{n=0}^{m}\Delta t\|\widehat{e}_{p}^{n+1/2}\|_{l^2,M}^2\\
\leq &C(\Delta t^4+h^4+k^4),\quad   \ m\leq N.
\endaligned
\end{equation}
Thus we have
\begin{equation}\label{e_error_estimate25}
\aligned
&\|Z^{m+1}-\phi^{m+1}\|_{l^2,M}+\|\textbf{D}Z^{m+1}-\textbf{D}\phi^{m+1}\|_{l^2}+
|R^{m+1}-r^{m+1}|\\
&+\left(\sum_{n=0}^{m}\Delta t\|\textbf{D}W^{n+1/2}-\textbf{D}\mu^{n+1/2}\|_{l^2}^2\right)^{1/2}\\
&+\left(\sum_{n=0}^{m}\Delta t\|W^{n+1/2}-\mu^{n+1/2}\|_{l^2,M}^2\right)^{1/2}\\
\leq&C(\|\phi\|_{W^{1}_{\infty}(J;W^{4}_{\infty}(\Omega))}+\|\mu\|_{L^{\infty}(J;W^{4}_{\infty}(\Omega))} )(h^2+k^2)\\
&+C\|\phi\|_{W^{3}_{\infty}(J;W^{1}_{\infty}(\Omega))}\Delta t^2.
\endaligned
\end{equation}
Recalling Lemma \ref{le_auxiliary}, we can obtain that
\begin{equation}\label{e_error_estimate26}
\aligned
\|d_x(U^{m}_1-{u}^{m}_1)\|_{l^2,M}+\|d_y(U^{m}_2-{u}^{m}_2)\|_{l^2,M}
\leq O(\Delta t^2+h^2+k^2),
\endaligned
\end{equation}
\begin{equation}\label{e_error_estimate27}
\aligned
\|U^{m}_1-{u}^{m}_1\|_{l^2,T,M}+&\|U^{m}_2-{u}^{m}_2\|_{l^2,M,T}+\left(\sum\limits_{l=1}^{m}\Delta t\|(P-p)^{l-1/2}\|^2_{l^2,M}\right)^{1/2}\\
\leq &O(\Delta t^2+h^2+k^2),
\endaligned
\end{equation}
\begin{equation}\label{e_error_estimate28}
\aligned
& \|D_y(U^{m}_1-{u}^{m}_1)\|_{l^2,T_y}\leq O(\Delta t^2+h^2+k^{3/2}),
\endaligned
\end{equation}
\begin{equation}\label{e_error_estimate29}
\aligned
&\|D_x(U^{m}_2-{u}^{m}_2)\|_{l^2,T_x}\leq O(\Delta t^2+h^{3/2}+k^2).
\endaligned
\end{equation}

Combing the above results together, we finally obtain our main results:
\begin{theorem}\label{thm: error_estimates}
Assuming $\phi\in W^{1}_{\infty}(J;W^{4}_{\infty}(\Omega)) \cap W^{3}_{\infty}(J;W^{1}_{\infty}(\Omega)),\mu\in L^{\infty}(J;W^{4}_{\infty}(\Omega))$, $\textbf{u}\in W^{3}_{\infty}(J;W^{4}_{\infty}(\Omega))^2$, $p\in W^{3}_{\infty}(J;W^{3}_{\infty}(\Omega))$ and $\Delta t\leq C(h+k)$, then for the Cahn-Hilliard-Stokes system, there exists a positive constant $C$ independent of $h$, $k$ and $\Delta t$ such that
\begin{equation}\label{e_error_estimate1}
\aligned
&\|Z^{m+1}-\phi^{m+1}\|_{l^2,M}+\|\textbf{D}Z^{m+1}-\textbf{D}\phi^{m+1}\|_{l^2}+
|R^{m+1}-r^{m+1}|\\
&+\left(\sum_{n=0}^{m}\Delta t\|\textbf{D}W^{n+1/2}-\textbf{D}\mu^{n+1/2}\|_{l^2}^2\right)^{1/2}\\
&+\left(\sum_{n=0}^{m}\Delta t\|W^{n+1/2}-\mu^{n+1/2}\|_{l^2,M}^2\right)^{1/2}\\
\leq&C(\|\phi\|_{W^{1}_{\infty}(J;W^{4}_{\infty}(\Omega))}+\|\mu\|_{L^{\infty}(J;W^{4}_{\infty}(\Omega))} )(h^2+k^2)\\
&+C\|\phi\|_{W^{3}_{\infty}(J;W^{1}_{\infty}(\Omega))}\Delta t^2,\quad   \ m\leq N,
\endaligned
\end{equation}

\begin{equation}\label{e_error_estimate26***}
\aligned
\|d_x(U^{m}_1-{u}^{m}_1)\|_{l^2,M}+\|d_y(U^{m}_2-{u}^{m}_2)\|_{l^2,M}
\leq O(\Delta t^2+h^2+k^2), \quad   \ m\leq N,
\endaligned
\end{equation}
\begin{equation}\label{e_error_estimate27***}
\aligned
\|\textbf{U}^{m}-\textbf{u}^{m}\|_{l^2}+\left(\sum\limits_{l=1}^{m}\Delta t\|(P-p)^{l-1/2}\|^2_{l^2,M}\right)^{1/2}\leq &O(\Delta t^2+h^2+k^2),\quad   \ m\leq N,
\endaligned
\end{equation}
\begin{equation}\label{e_error_estimate28***}
\aligned
& \|D_y(U^{m}_1-{u}^{m}_1)\|_{l^2,T_y}\leq O(\Delta t^2+h^2+k^{3/2}),\quad   \ m\leq N,
\endaligned
\end{equation}
\begin{equation}\label{e_error_estimate29***}
\aligned
&\|D_x(U^{m}_2-{u}^{m}_2)\|_{l^2,T_x}\leq O(\Delta t^2+h^{3/2}+k^2),\quad   \ m\leq N.
\endaligned
\end{equation}
\end{theorem}
\medskip

 \section{Numerical experiments} \label{Numerical experiments}
In this section we provide some 2-D numerical experiments to gauge the SAV/CN-FD method developed in the previous sections.

We transform (\ref{energy1}) as
\begin{equation}\label{definition of energy_transformed}
\aligned
E(\phi)=\int_{\Omega}\{\frac 1 2|\textbf{u}|^2+\lambda(\frac{1}{2}|\nabla \phi |^2+\frac{\beta}{2\epsilon^2}\phi^2
+\frac{1}{4\epsilon^2}(\phi^2-1-\beta)^2 -\frac{\beta^2+2\beta}{4\epsilon^2})\}d\textbf{x},
\endaligned
\end{equation}
where $\beta$ is a positive number to be chosen. To apply our scheme (\ref{e_model_r_full_discrete1})-(\ref{e_model_r_full_discrete6}) to the system (\ref{e_model}), we drop the constant in the free energy and specify $\displaystyle E_1(\phi)=\frac{1}{4\epsilon^2}\int_{\Omega}(\phi^2-1-\beta)^2d\textbf{x}$, and modify (\ref{e_model_r_full_discrete2}) into
\begin{equation}\label{e_model_r_full_discrete2_modify}
\aligned
W_{i+1/2,j+1/2}^{n+1/2}=&-\lambda[d_xD_xZ+d_yD_yZ]_{i+1/2,j+1/2}^{n+1/2}+
\frac{\lambda\beta}{\epsilon^2}Z_{i+1/2,j+1/2}^{n+1/2}\\
    &+\lambda\frac{R^{n+1/2}}{\sqrt{E_1^h(\tilde{Z}^{n+1/2})}}F^{\prime}(\tilde{Z}_{i,j}^{n+1/2}).
\endaligned
\end{equation}
Then we can obtain
\begin{equation}\label{e_numerical_1}
\aligned
F^{\prime}(\phi)=\frac{\delta E_1}{\delta \phi}=\frac{1}{\epsilon^2}\phi(\phi^2-1-\beta).
\endaligned
\end{equation}

For simplicity, we define
\begin{flalign*}
\renewcommand{\arraystretch}{1.5}
  \left\{
   \begin{array}{l}
\|f-g\|_{\infty,2}=\max\limits_{0\leq n\leq m}\left\{\|f^{n+q}-g^{n+q}\|_X\right\},\\
\|f-g\|_{2,2}=\left(\sum\limits_{n=0}^{m}\Delta t\left\|f^{n+q}-g^{n+q}\right\|_{X}^2\right)^{1/2},\\
\|R-r\|_{\infty}=\max\limits_{0\leq n\leq m}\{R^{n+1}-r^{n+1}\},\\
\end{array}\right.
\end{flalign*}
where $q=\frac{1}{2},~1$ and $X$ is the corresponding discrete $L^2$ norm. In the following simulations, we choose $ \Omega=(0,1)\times(0,1)$, $\beta=5$ and $\gamma=1$.

\subsection{Convergence rates of the SAV/CN-FD scheme for the Cahn-Hilliard-Navier-Stokes phase field model}

In this example 1, we take $T=0.1$, $\Delta t=1E-4$, $\lambda=0.1$, $\nu=0.1$, $\epsilon^2=0.1$, $M=0.001$, and the initial solution $\phi_0=\cos(\pi x)\cos(\pi y)$, $u_1(x,y)=-x^2(x-1)^2(y-1)(2y-1)y/128$ and $u_2(x,y)=-u_1(y,x)$. We measure Cauchy error to get around the fact that we do not have possession of exact solution. Specifically, the error between two different grid spacings $h$ and $\frac{h}{2}$ is calculated by $\|e_{\zeta}\|=\|\zeta_h-\zeta_{h/2}\|$.

 The numerical results are listed in Tables \ref{table1_example1}-\ref{table1_example3} and give solid supporting evidence for the expected second-order convergence of the SAV/CN-FD scheme for the Cahn-Hilliard-Navier-Stokes phase field model, which are consistent with the error estimates in Theorem \ref{thm: error_estimates}. Here we only present the results for $u_1$ since the results for $u_2$ are similar to $u_1$. 

\begin{table}[htbp]
\renewcommand{\arraystretch}{1.1}
\small
\centering
\caption{Errors and convergence rates of the phase function and auxiliary scalar function for example 1.}\label{table1_example1}
\begin{tabular}{p{1cm}p{1.5cm}p{0.7cm}p{1.8cm}p{0.7cm}p{1.8cm}p{0.7cm}}\hline
$h$    &$\|e_Z\|_{\infty,2}$    &Rate &$\|e_{\textbf{D}Z}\|_{\infty,2}$   &Rate  
&$\|e_{R}\|_{\infty}$    &Rate   \\ \hline
$1/10$     &3.09E-3                & ---    &1.37E-2         &---  &2.69E-5         &---\\
$1/20$      &7.74E-4                & 2.00    &3.43E-3         &1.99 &6.76E-6         &1.99\\
$1/40$      &1.93E-4                &2.00     &8.60E-4         &2.00 &1.69E-6         &2.00\\
$1/80$     &4.84E-5                &2.00    &2.15E-4         &2.00   &4.23E-7         &2.00\\
\hline
\end{tabular}
\end{table}

\begin{table}[htbp]
\renewcommand{\arraystretch}{1.1}
\small
\centering
\caption{Errors and convergence rates of the chemical potential and velocity for example 1.}\label{table1_example2}
\begin{tabular}{p{1cm}p{1.5cm}p{0.7cm}p{1.8cm}p{0.7cm}p{1.8cm}p{0.7cm}}\hline
$h$    &$\|e_{W}\|_{2,2}$    &Rate &$\|e_{\textbf{D}W}\|_{2,2}$   &Rate  
&$\|e_{\textbf{U}}\|_{\infty,2}$    &Rate   \\ \hline
$1/10$     &1.59E-3               & ---    &1.57E-2         &---  &1.67E-4         &---\\
$1/20$      &4.01E-4                & 1.98    &4.09E-3         &1.94 &3.67E-5         &2.19\\
$1/40$      &1.01E-4                &2.00     &1.03E-3         &1.99 &8.88E-6         &2.05\\
$1/80$     &2.51E-5                &2.00    &2.59E-4         &2.00   &2.20E-6         &2.01\\
\hline
\end{tabular}
\end{table}

\begin{table}[htbp]
\renewcommand{\arraystretch}{1.1}
\small
\centering
\caption{Errors and convergence rates of the velocity and pressure for example 1.}\label{table1_example3}
\begin{tabular}{p{0.7cm}p{1.8cm}p{0.7cm}p{1.8cm}p{0.7cm}p{1.7cm}p{0.7cm}}\hline
$h$    &$\|e_{d_xU_1}\|_{\infty,2}$    &Rate &$\|e_{D_yU_1}\|_{\infty,2}$   &Rate  
&$\|e_{P}\|_{2,2}$    &Rate   \\ \hline
$1/10$     &9.14E-4                & ---    &1.54E-3         &---  &1.06E-3         &---\\
$1/20$      &2.05E-4                & 2.16    &4.28E-4         &1.85 &2.63E-4         &2.01\\
$1/40$      &4.99E-5                &2.04     &1.36E-4         &1.66 &6.56E-5         &2.00\\
$1/80$     &1.24E-5                &2.01   &4.56E-5         &1.57   &1.64E-5         &2.00\\
\hline
\end{tabular}
\end{table}

\subsection{The dynamics of a square shape fluid}
In this example 2, the evolution of a square shaped fluid bubble is simulated by using the following parameters:
$$ \epsilon=0.01,\ \nu=1,\ \lambda=0.01,\ M=0.002,\  \hat{h}=1/100,\ \Delta t=1E-3.$$
The initial velocity and pressure are set to zero. The initial phase function is chosen to be a rectangular bubble, i.e., $\phi=1$ inside the bubble and $\phi=-1$ outside the bubble. Snapshots of the phase evolution at time $t=0,5,6,8,10,$ respectively are presented in Fig. \ref{fig_ square shape}. As we can see, the rectangular bubble deforms into a circular bubble due to the surface tension.
\begin{figure}[!htp]
\centering
\includegraphics[scale=0.25]{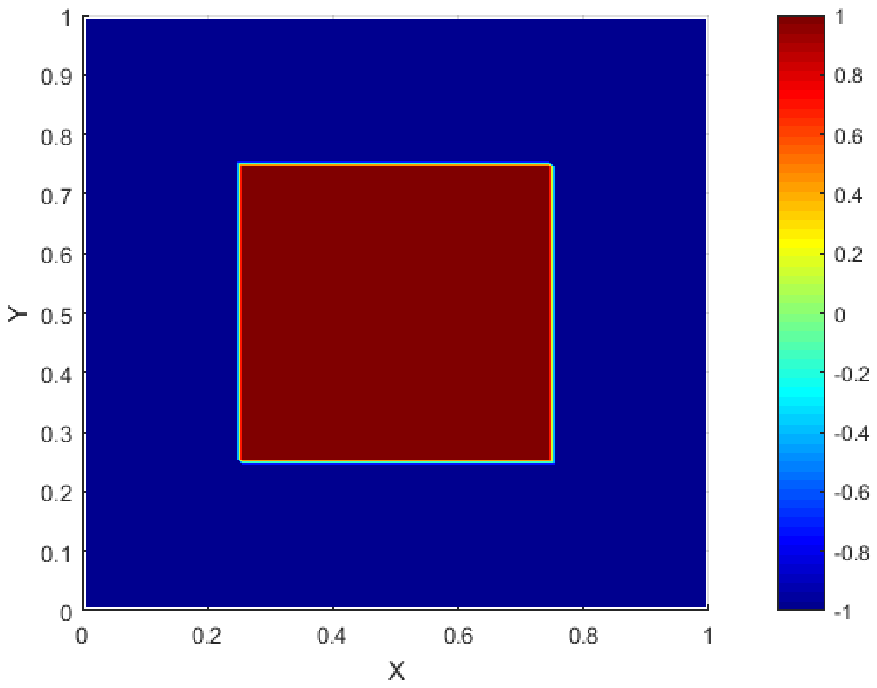}
\includegraphics[scale=0.25]{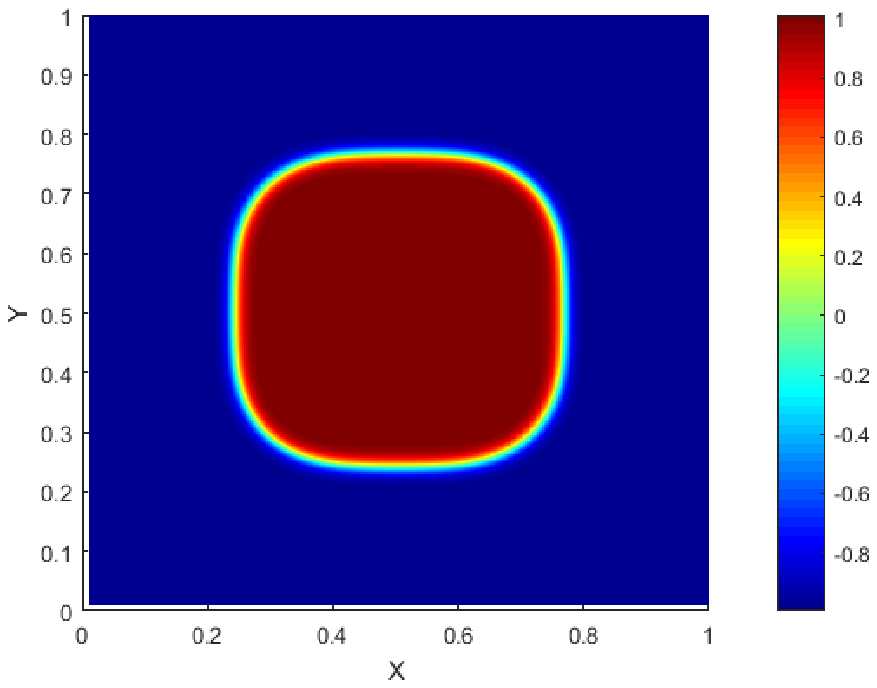}
\includegraphics[scale=0.25]{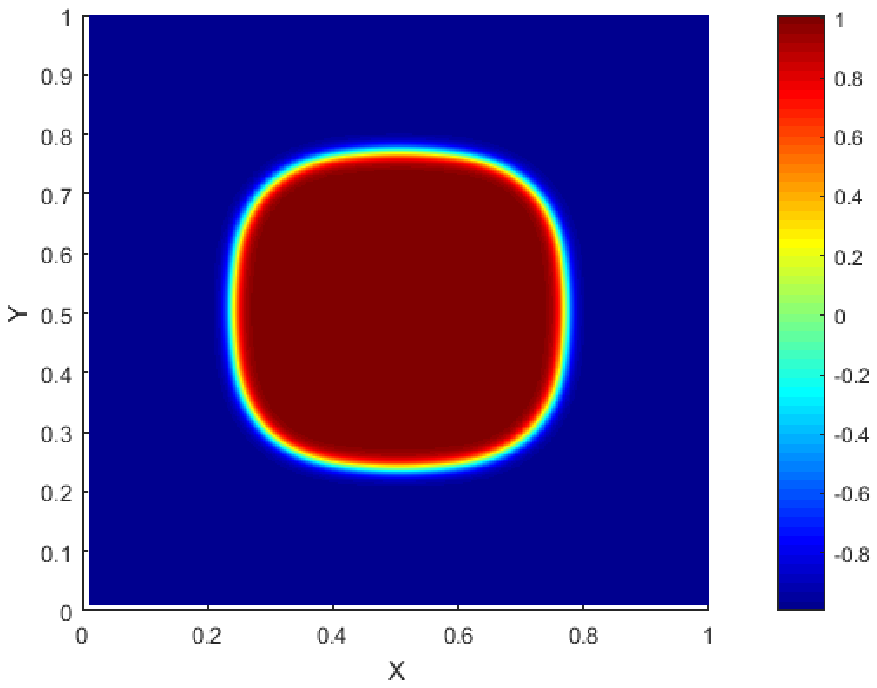}
\includegraphics[scale=0.25]{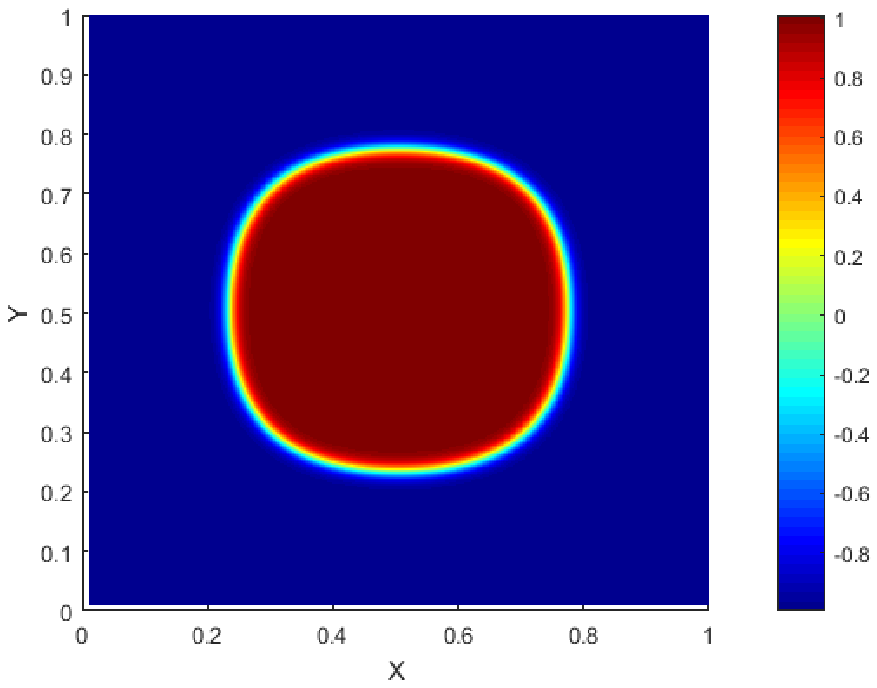}
\includegraphics[scale=0.25]{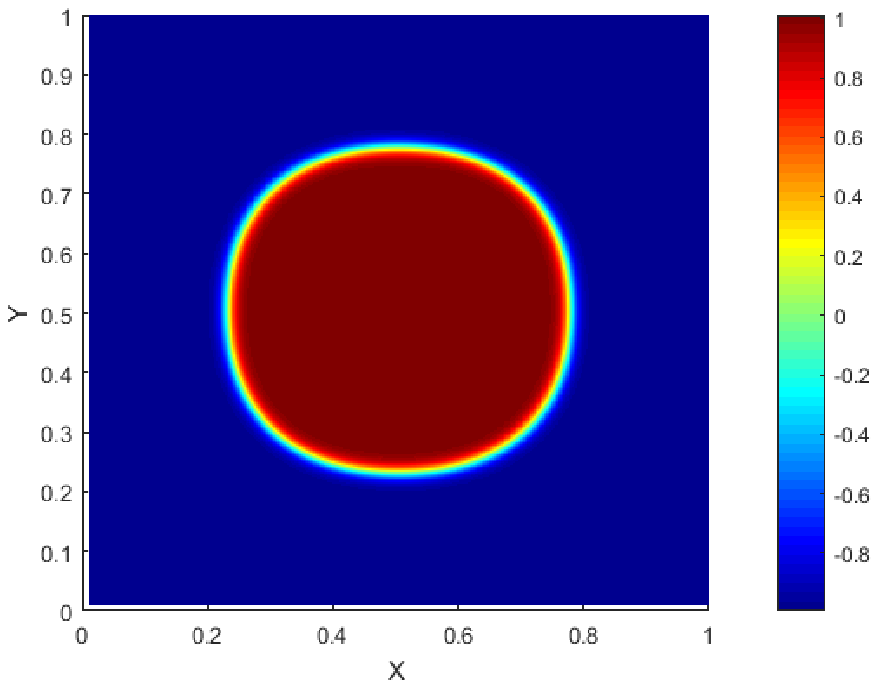}
\caption{Snapshots of the phase function in example 2 at $t=0,5,6,8,10,$ respectively.} \label{fig_ square shape}
\end{figure}

\subsection{Buoyancy-driven flow}
In this example 2, as the test of buoyancy-driven flow, we consider the case of a single bubble rising in a rectangular box. Similar to \cite{collins2013efficient}, we modify the Navier-Stokes equation (\ref{e_modelC}) as follows:
\begin{equation}\label{e_numerical_2}
\aligned
 \frac{\partial \textbf{u}}{\partial t}+\textbf{u}\cdot \nabla\textbf{u}
     -\nu\Delta\textbf{u}+\nabla p=\mu\nabla\phi+\textbf{b},
\endaligned
\end{equation}
where $\textbf{b}$ is a buoyancy term that depends on the mass density $\rho$. We assume that the mass density depends on $\phi$, and the following Boussinesq type approximation is applied:
\begin{equation}\label{e_numerical_3}
\aligned
\textbf{b}=(0,-b(\phi))^{t},\ b(\phi)=\chi(\phi-\phi_0),
\endaligned
\end{equation}
where $\phi_0$ is a constant (usually the average value of $\phi$), and $\chi$ is a constant. In this example, the numerical and physical parameters are as follows:
\begin{flalign*}
\renewcommand{\arraystretch}{1.5}
  \left\{
   \begin{array}{l}
  \hat{h}=1/100,\  \Delta t=5E-4,\ M=0.01,\\
  \epsilon=0.01,\ \nu=1,\ \lambda=0.001,\\
  \phi_0=-0.05,\ \chi=40.
\end{array}\right.
\end{flalign*}
The initial condition for the phase function is choose to be a circular bubble that centered at $(\frac{1}{2},\frac{1}{4})$, and the initial data for the velocity is taken as $\textbf{u}^0=0$. Snapshots of the phase evolution at time $t=0.5,1,4,4.1,4.2,5$ respectively are presented in Fig. \ref{fig1_example5}. It starts as a circular bubble near the bottom of the domain. The density of the bubble is lighter than the density of the surrounding fluid. As expected, the bubble rises, reaching an elliptical shape, and then deforms as it approaches the upper boundary.


\begin{figure}[!htp]
\centering
\includegraphics[scale=0.43]{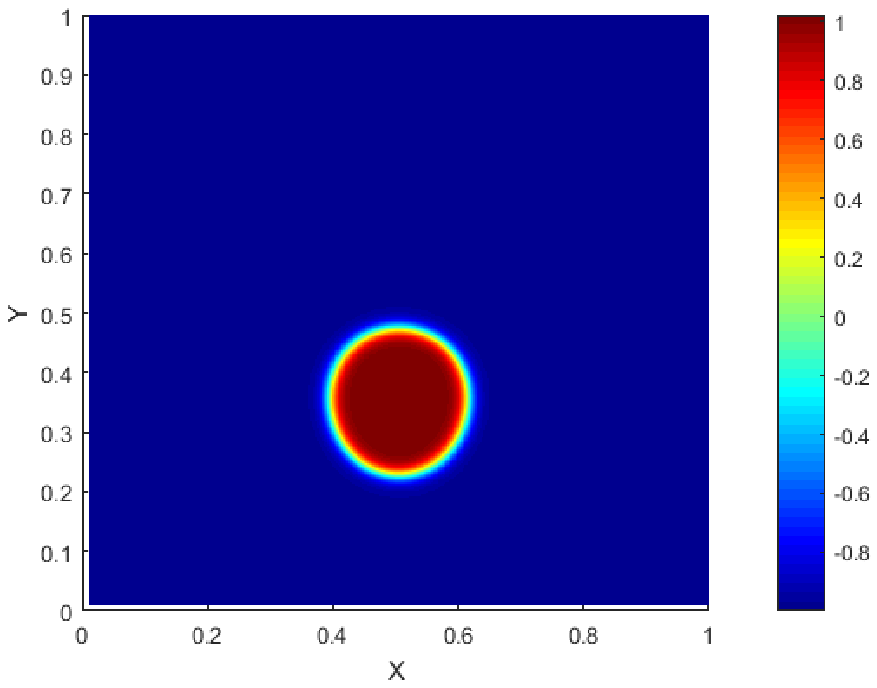}
\includegraphics[scale=0.43]{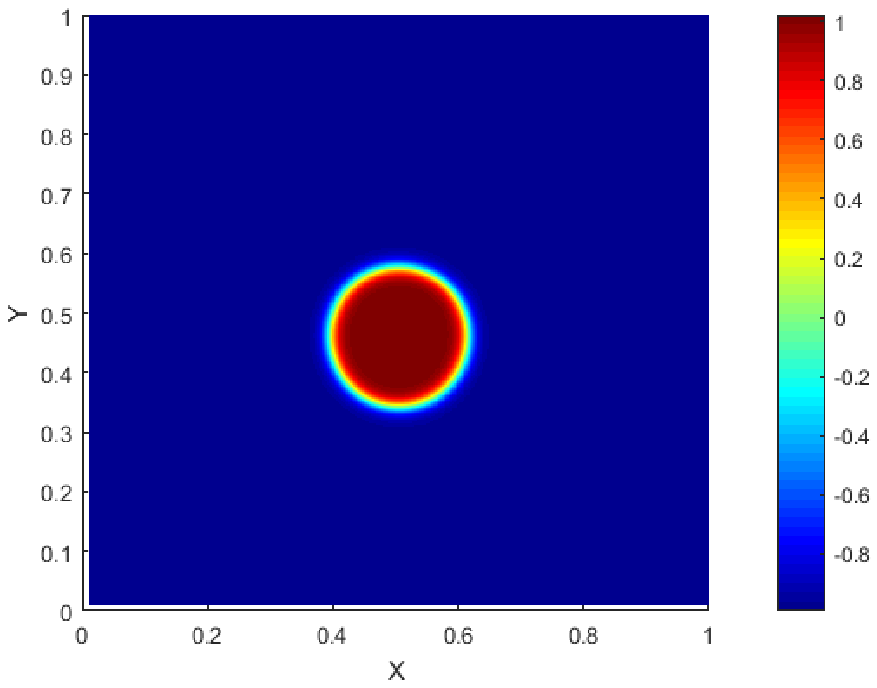}
\includegraphics[scale=0.43]{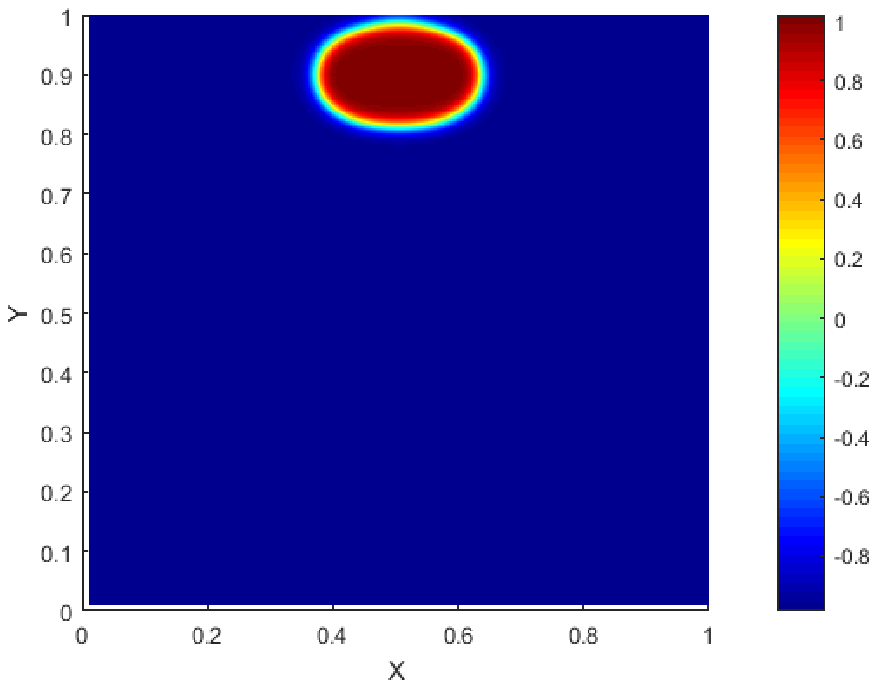}
\includegraphics[scale=0.43]{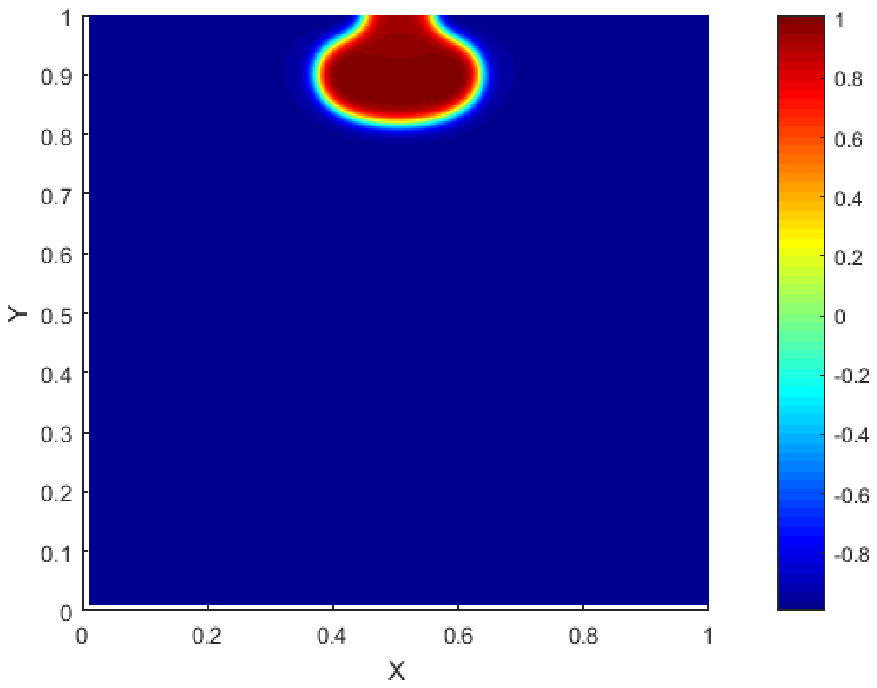}
\includegraphics[scale=0.43]{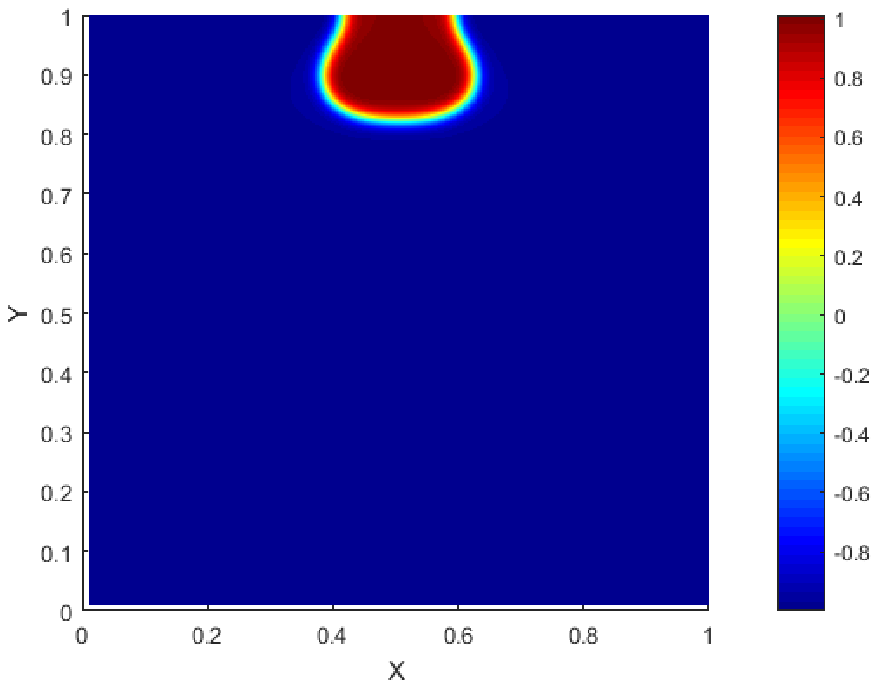}
\includegraphics[scale=0.43]{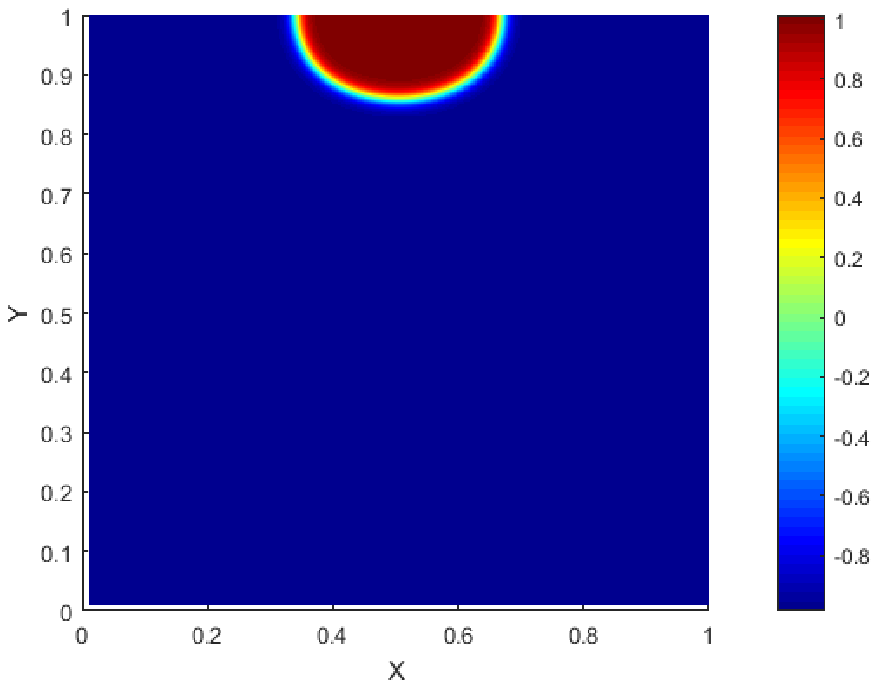}
\caption{Snapshots of the phase function in example 2 at $t=0.5,1,4,4.1,4.2,5$ respectively.} \label{fig1_example5}
\end{figure}

\section{Conclusion}
We developed a second-order fully discrete SAV-MAC scheme for the Cahn-Hilliard-Navier-Stokes phase field model, and proved that it is unconditionally energy stable. We also carried out a rigorous error analysis for the Cahn-Hilliard-Stokes system and derived   second-order error estimates both in time and space for phase field variable, chemical potential, velocity and pressure in different discrete norms. 


 The SAV-MAC scheme, with an explicit treatment of the convective term in the phase equation,  is  extremely efficient as it leads to, at each time step, a sequence of Poisson type equations that can be solved by using fast Fourier transforms. We provided several numerical results to 
 demonstrate the robustness and accuracy of the SAV-MAC scheme for the Cahn-Hilliard-Navier-Stokes phase field model.

 We only carried out an error analysis for the Cahn-Hilliard-Stokes system. To derive corresponding error estimates  for the Cahn-Hilliard-Navier-Stokes system, one needs to use new discretizing techniques such as a high order upwind method to deal with the nonlinear term. This will be a subject of future research.

\section*{Appendix\ A\ \ Finite difference discretization on the staggered grids}
\renewcommand\thesection{A}
To fix the idea, we consider $\Omega=(L_{lx},L_{rx})\times (L_{ly},L_{ry})$. Three dimensional rectangular domains can be dealt with similarly.

The two dimensional domain $\Omega$ is partitioned by  $\Omega_x\times \Omega_y$,    where
  \begin{eqnarray}
    & & \Omega_x: L_{lx}=x_{0}<x_{1}<\cdots<x_{N_x-1}<x_{N_x}=L_{rx},\nonumber\\
    & & \Omega_y: L_{ly}=y_{0}<y_{1}<\cdots<y_{N_y-1}<y_{N_y}=L_{ry}.\nonumber
   \end{eqnarray}
For simplicity we also use the following notations:
\begin{equation}    \label{eq:grid-on-bound}
\left\{
\begin{array}{ll}
x_{-1/2}=x_{0}=L_{lx},& x_{N_x+1/2}=x_{N_x}=L_{rx},\\
y_{-1/2}=y_{0}=L_{ly},& y_{N_y+1/2}=y_{N_y}=L_{ry}.
\end{array}
\right.
\end{equation}
For possible integers $i,j$, $0\leq i\leq N_x,\ 0\leq j\leq N_y$, define
    \begin{eqnarray}
& & x_{i+1/2}=\frac{x_{i}+x_{i+1}}{2},\quad  h_{i+1/2}=x_{i+1}-x_{i},\quad   h=\max\limits_{i}h_{i+1/2},\nonumber\\
& & h_{i}=x_{i+1/2}-x_{i-1/2}=\frac{h_{i+1/2}+h_{i-1/2}}{2},\nonumber\\
& & y_{j+1/2}=\frac{y_{j}+y_{j+1}}{2},\quad  k_{j+1/2}=y_{j+1}-y_{j}, \quad   k=\max\limits_{j}k_{j+1/2},\nonumber\\
& & k_{j}=y_{j+1/2}-y_{j-1/2}=\frac{k_{j+1/2}+k_{j-1/2}}{2},\nonumber\\
& & \Omega_{i+1/2,j+1/2}=(x_{i},x_{i+1})\times (y_{j},y_{j+1}).\nonumber
    \end{eqnarray}
It is clear that
$$ \displaystyle
h_{0}=\frac{h_{1/2}}{2},\ h_{N_x}=\frac{h_{N_x-1/2}}{2} ,\ \ k_{0}=\frac{k_{1/2}}{2},\ k_{N_y}=\frac{k_{N_y-1/2}}{2}.$$
For a function $f(x,y)$, let $f_{l,m}$ denote $f(x_l,y_m)$ where $l$ may take values $i,\ i+1/2$ for  integer $i$, and $m$ may take values $j,\ j+1/2$ for  integer $j$. For discrete functions with values at proper nodal-points, define
\begin{equation}
\left\{
\begin{array}{lll}
  \displaystyle [d_{x}f]_{i+1/2,m}=\frac{f_{i+1,m}-f_{i,m}}{h_{i+1/2}},\qquad & \displaystyle  [D_{y}f]_{l,j+1}=\frac{f_{l,j+3/2}-f_{l,j+1/2}}{k_{j+1}}, \\
 \displaystyle [D_{x}f]_{i+1,m}=\frac{f_{i+3/2,m}-f_{i+1/2,m}}{h_{i+1}}, &  \displaystyle  [d_{y}f]_{l,j+1/2}=\frac{f_{l,j+1}-f_{l,j}}{k_{j+1/2}}.
\end{array}
\right. \label{def:difference}
\end{equation}
For functions $f$ and $g$, define some discrete $l^2$ inner products and norms as follows.
\begin{eqnarray}
 (f,g)_{l^2,M} &\equiv &  \sum\limits_{i=0}^{N_x-1}\sum\limits_{j=0}^{N_y-1} h_{i+1/2}k_{j+1/2} f_{i+1/2,j+1/2} g_{i+1/2,j+1/2}, \label{inner:l2-M}\\
      (f,g)_{l^2,T_x} &\equiv &  \sum\limits_{i=0}^{N_x}\sum\limits_{j=1}^{N_y-1} h_{i}k_{j} f_{i,j} g_{i,j}, \label{inner:l2_x}\\
  (f,g)_{l^2,T_y} &\equiv &  \sum\limits_{i=1}^{N_x-1}\sum\limits_{j=0}^{N_y} h_{i}k_{j} f_{i,j} g_{i,j}, \label{inner:l2_y}\\
   \|f\|_{l^2,\xi}^2 &\equiv &  (f,f)_{l^2,\xi},\qquad \xi=M,\ T_x,\ T_y   . \label{norm:l2}
 \end{eqnarray}
Further define discrete $l^2$ inner products and norms as follows.
\begin{eqnarray}
& &(f,g)_{l^2,T,M} \equiv \sum\limits_{i=1}^{N_x-1}\sum\limits_{j=0}^{N_y-1} h_{i}k_{j+1/2} f_{i,j+1/2} g_{i,j+1/2}, \label{inner:l2-T-M}\\
& & (f,g)_{l^2,M,T} \equiv \sum\limits_{i=0}^{N_x-1}\sum\limits_{j=1}^{N_y-1} h_{i+1/2}k_{j} f_{i+1/2,j}g_{i+1/2,j}, \label{inner:l2-M-T}\\
& &  \|f\|_{l^2,T,M}^2 \equiv   (f,f)_{l^2,T,M}, \quad  \|f\|_{l^2,M,T}^2 \equiv  (f,f)_{l^2,M,T}. \label{norm:l2-T-M}
 \end{eqnarray}
For vector-valued functions $\textbf{u}=(u_1,u_2)$, it is clear that
 \begin{eqnarray}
 \|d_x u_{1}\|_{l^2,M}^2 &\equiv &  \sum\limits_{i=0}^{N_x-1}\sum\limits_{j=0}^{N_y-1}h_{i+1/2}k_{j+1/2} |d_x u_{1,i+1/2,j+1/2}|^2, \label{norm:d-x-u-x}\\
 \|D_y u_{1}\|_{l^2,T_y}^2 &\equiv &     \sum\limits_{i=1}^{N_x-1}\sum\limits_{j=0}^{N_y}h_{i}k_{j} |D_y u_{1,i,j}|^2,   \label{norm:d-y-u-x}
 \end{eqnarray}
and $\|d_y u_{2}\|_{l^2,M},\ \|D_x u_{2}\|_{l^2,T_x}$ can be represented similarly.
Finally define the discrete $H^1$-norm and discrete $l^2$-norm of a vectored-valued function $\textbf{u}$,
 \begin{eqnarray}
  \|D \textbf{u}\|^2  & \equiv & \|d_x u_{1}\|_{l^2,M}^2+ \|D_y u_{1}\|_{l^2,T_y}^2+  \|D_x u_{2}\|_{l^2,T_x}^2+\|d_y u_{2}\|_{l^2,M}^2. \label{norm:d-u}\\
  \|\textbf{u}\|_{l^2}^2 &  \equiv & \| u_{1}\|_{l^2,T,M}^2+ \|  u_{2}\|_{l^2,M,T}^2. \label{norm:l2-u}
 \end{eqnarray}
 For simplicity we only consider the case that for all $h_{i+1/2}=h,\ k_{j+1/2}=k$, i.e. uniform meshes are used both in $x$ and $y$-directions. 

\bibliographystyle{siamplain}
\bibliography{SAV_Navier_Stokes_CH}

\end{document}